\documentclass[11pt]{article}
\usepackage{amsmath,amssymb,amsthm,amscd}

\newtheorem{theorem}{Theorem}[section]
\newtheorem{lemma}[theorem]{Lemma}
\newtheorem{proposition}[theorem]{Proposition}
\newtheorem{corollary}[theorem]{Corollary}

\theoremstyle{plain}

\theoremstyle{definition}
\newtheorem{definition}[theorem]{Definition}

\numberwithin{equation}{section}

\renewcommand{\labelenumi}{\textup{(\theenumi)}}

\newcommand{\Homeo}{\operatorname{Homeo}}

\newcommand{\id}{\operatorname{id}}

\newcommand{\Ker}{\operatorname{Ker}}

\newcommand{\Ad}{\operatorname{Ad}}

\newcommand{\N}{\mathbb{N}}
\newcommand{\Z}{\mathbb{Z}}
\newcommand{\T}{\mathbb{T}}

\newcommand{\Zp}{{\mathbb{Z}}_+}

\title{ 
Flip conjugacy and asymptotic continuous orbit equivalence of Smale spaces 
}
\author{Kengo Matsumoto \\
Department of Mathematics \\
Joetsu University of Education \\
Joetsu, 943-8512, Japan
}

\date{ }

\begin{document}
\maketitle

\def\det{{{\operatorname{det}}}}

\begin{abstract} 
We study asymptotic continuous orbit equivalence of Smale spaces. 
We prove that two irreducible Smale spaces are flip conjugate if and only if 
there exists a periodic point preserving homeomorphism 
giving an asymptotic continuous orbit equivalence
between them.
We introduce a notion of asymptotic topological conjugacy and asymptotic flip
conjugacy in Smale spaces 
and characterize them in terms of the associated Ruelle algebras 
with dual actioons.
We finally characterize the flip conjugacy classes of irreducible two-sided topological Markov shifts 
in terms of the associated Ruelle algebras with its $C^*$-subalgebras.  
\end{abstract}



\def\R{{\mathcal{R}}}
\def\RT{\tilde{\mathcal{R}}}

\def\OA{{{\mathcal{O}}_A}}
\def\OB{{{\mathcal{O}}_B}}
\def\RAsu{{{\mathcal{R}}_A^{s,u}}}
\def\RAC{{{\mathcal{R}}_A^\circ}}
\def\RAR{{{\mathcal{R}}_A^\rho}}
\def\RtA{{{\mathcal{R}}_{A^t}}}
\def\R{{\mathcal{R}}}
\def\RB{{{\mathcal{R}}_B}}
\def\FA{{{\mathcal{F}}_A}}
\def\FB{{{\mathcal{F}}_B}}
\def\FA{{{\mathcal{F}}_A}}
\def\FtA{{{\mathcal{F}}_{A^t}}}
\def\FB{{{\mathcal{F}}_B}}
\def\DtA{{{\mathcal{D}}_{A^t}}}
\def\FDA{{{\frak{D}}_A^{s,u}}}
\def\FDtA{{{\frak{D}}_{A^t}}}
\def\FFA{{{\frak{F}}_A^{s,u}}}
\def\FFtA{{{\frak{F}}_{A^t}}}
\def\FFB{{{\frak{F}}_B}}
\def\OZ{{{\mathcal{O}}_Z}}
\def\V{{\mathcal{V}}}
\def\E{{\mathcal{E}}}
\def\OtA{{{\mathcal{O}}_{A^t}}}
\def\SOA{{{\mathcal{O}}_A}\otimes{\mathcal{K}}}
\def\SOB{{{\mathcal{O}}_B}\otimes{\mathcal{K}}}
\def\SOZ{{{\mathcal{O}}_Z}\otimes{\mathcal{K}}}
\def\SOtA{{{\mathcal{O}}_{A^t}\otimes{\mathcal{K}}}}
\def\DA{{{\mathcal{D}}_A}}
\def\DB{{{\mathcal{D}}_B}}
\def\DZ{{{\mathcal{D}}_Z}}

\def\SDA{{{\mathcal{D}}_A}\otimes{\mathcal{C}}}
\def\SDB{{{\mathcal{D}}_B}\otimes{\mathcal{C}}}
\def\SDZ{{{\mathcal{D}}_Z}\otimes{\mathcal{C}}}
\def\SDtA{{{\mathcal{D}}_{A^t}\otimes{\mathcal{C}}}}
\def\BC{{{\mathcal{B}}_C}}
\def\BD{{{\mathcal{B}}_D}}
\def\OAG{{\mathcal{O}}_{A^G}}
\def\OBG{{\mathcal{O}}_{B^G}}
\def\O2{{{\mathcal{O}}_2}}
\def\D2{{{\mathcal{D}}_2}}

\newcommand{\mathP}{\mathcal{P}}
\def\OAG{{\mathcal{O}}_{A^G}}
\def\DAG{{\mathcal{D}}_{A^G}}
\def\OBG{{\mathcal{O}}_{B^G}}
\def\CKT{{\mathcal{T}}}
\def\Max{{{\operatorname{Max}}}}
\def\Per{{{\operatorname{Per}}}}
\def\PerB{{{\operatorname{PerB}}}}
\def\Homeo{{{\operatorname{Homeo}}}}
\def\HA{{{\frak H}_A}}
\def\HB{{{\frak H}_B}}
\def\HSA{{H_{\sigma_A}(X_A)}}
\def\Out{{{\operatorname{Out}}}}
\def\Aut{{{\operatorname{Aut}}}}
\def\Ad{{{\operatorname{Ad}}}}
\def\Inn{{{\operatorname{Inn}}}}
\def\Int{{{\operatorname{Int}}}}
\def\det{{{\operatorname{det}}}}
\def\exp{{{\operatorname{exp}}}}
\def\cobdy{{{\operatorname{cobdy}}}}
\def\Ker{{{\operatorname{Ker}}}}
\def\ind{{{\operatorname{ind}}}}
\def\id{{{\operatorname{id}}}}
\def\supp{{{\operatorname{supp}}}}
\def\co{{{\operatorname{co}}}}
\def\Sco{{{\operatorname{Sco}}}}
\def\flip{{{\operatorname{flip}}}}
\def\U{{{\mathcal{U}}}}
\def\Porb{{{\operatorname{P}_{\operatorname{orb}}}}}
\def\ACOE{\underset{\operatorname{ACOE}}{\sim}}
\def\StACOE{\overset{St}{\underset{\operatorname{ACOE}}{\sim}}}
\def\UniACOE{\overset{Uni}{\underset{\operatorname{ACOE}}{\sim}}}
\def\PStACOE{\overset{St}{\underset{\operatorname{+ACOE}}{\sim}}}
\def\NStACOE{\overset{St}{\underset{\operatorname{-ACOE}}{\sim}}}
\def\bs{\bar{\sigma}}
\def\bsA{\bar{\sigma}_A}
\def\bsB{\bar{\sigma}_B}
\def\bXA{\bar{X}_A}
\def\bXB{\bar{X}_B}



\section{ Introduction}
There are many interesting and important researches dealing with  interplay between classification of topological dynamical systems by orbit equivalence 
and classification of the associated $C^*$-algebras.
The first important result in this direction is the work by Giordano--Putnam--Skau
(\cite{GPS1995}, \cite{GMPS}, \cite{HPS}, etc. ).
They obtained a fundamental result on minimal homeomorphisms on Cantor sets,
which  says that two minimal homeomorphisms on Cantor sets are strongly orbit equivalent if and only if 
their associated crossed product $C^*$-algebras are isomorphic (\cite{GPS1995}).
J. Tomiyama and Boyle--Tomiyama  studied a generalization of the GPS's results (\cite{BT}, \cite{ToPacific}, etc. ).
H. Matui and the author in \cite{MMKyoto}
obtained a classification result about 
continuous orbit equivalence of one-sided topological Markov shifts,
that are continuous surjections but not homeomorphisms.
Two-sided topological Markov shifts are typical examples of 
hyperbolic dynamical systems called Smale spaces, that was introduced by D. Ruelle in \cite{Ruelle1},\cite{Ruelle2}.
After Ruelle's work, 
I. Putnam in \cite{Putnam1} initiated to study structure of the $C^*$-algebras 
associated to groupoids constructed 
from Smale spaces (cf. \cite{Putnam2}, \cite{PutSp}, etc.).
Putnam in \cite{Putnam1} first defined three kinds of equivalence relations
called stable, unstable and asymptotic equivalence relations on
a Smale space $(X,\phi)$.
Their associated groupoids are written
$G_\phi^s, G_\phi^u$ and $G_\phi^a$, respectively.
He also studied the groupoids 
$G_\phi^s\rtimes\Z, G_\phi^u\rtimes\Z$ and $G_\phi^a\rtimes\Z$
of their semidirect products by the integer group $\Z$ and 
their associated $C^*$-algebras.
The $C^*$-algebras are written
$\R_\phi^s, \R_\phi^u$ and $\R_\phi^a,$ respectively
and called the stable Ruelle algebra, the unstable Ruelle algebra 
and the asymptotic Ruelle algebra.
The asymptotic groupoid $G_\phi^a\rtimes\Z$ is \'etale, 
whereas the other two  
$G_\phi^s\rtimes\Z, G_\phi^u\rtimes\Z$ are not.
Its $C^*$-algebra $\R_\phi^a$ is a unital nuclear simple $C^*$-algebra with unique tracial state if $(X,\phi)$ is irreducible.    
The unit space $(G_\phi^a\rtimes\Z)^{(0)}$ 
of the groupoid
$G_\phi^a\rtimes\Z$
is naturally identified with the original space $X.$
 The irreducibility condition ensures that the \'etale groupoid 
 $G_\phi^a\rtimes\Z$ is essentially principal so that
the commutative $C^*$-algebra $C(X)$ on the original space $X$ is regarded as a maximal abelian $C^*$-subalgebra of $\R_\phi^a$
 because $C(X)$ is canonically isomorphic to the commutative $C^*$-subalgebra  $C((G_\phi^a\rtimes\Z)^{(0)})$ on the unit space $(G_\phi^a\rtimes\Z)^{(0)}$
of the groupoid $G_\phi^a\rtimes\Z.$
In this paper, we will focus on the \'etale groupoid $G_\phi^a\rtimes\Z$
 $$
G_\phi^a \rtimes\Z = \{(x,n,z) \in X \times \Z \times X \mid (\phi^n(x), z) \in G_\phi^a\}.
$$

In \cite{MaCJM}, 
the author introduced a notion of asymptotic continuous orbit equivalence of Smale spaces.
The notion of asymptotic continuous orbit equivalence is, roughly speaking,
a continuous orbit equivalence in Smale spaces up to asymptotic equivalence.
It comes from an isomorphism between the associated \'etale groupoids. 
Hence by Renault's result \cite{Renault2} (cf. \cite{Renault1}, \cite{Renault3}),  
we know that two Smale spaces 
$(X,\phi)$ and $(Y,\psi)$ are asymptotically continuous orbit equivalent if and only if
the pairs $(\R_\phi^a, C(X))$ and $(\R_\psi^a, C(Y))$ are isomorphic.
In this paper, we will continue to study asymptotic continuous orbit equivalence in Smale spaces.
We will first  reformulate the original definition
of asymptotic continuous orbit equivalence in Smale spaces
given in \cite[Definition 3.2]{MaCJM}
into a slightly simpler form than the original one \cite[Definition 3.2]{MaCJM}
(Proposition \ref{prop:main1}).

Although topologically conjugate Smale spaces induce their asymptotic continuous orbit equivalence,
asymptotic continuous orbit equivalence between Smale spaces does not necessarily
imply topological conjugacy.
Two topological dynamical systems $(X,\phi)$ and $(Y,\psi)$ 
are said to be flip conjugate if 
$(X,\phi)$ is topologically conjugate to $(Y,\psi)$ or its inverse $(Y,\psi^{-1}).$ 
By \cite[Proposition 11.1]{MaCJM}, an irreducible Smale space $(X,\phi)$ 
is asymptotically continuous orbit equivalent to its inverse $(X,\phi^{-1})$.
Hence  
if irreducible Smale spaces $(X, \phi)$ and $(Y, \psi)$
are flip conjugate,
then they are asymptotically continuous orbit equivalent.

As there is an example of irreducible Smale space $(X,\phi)$
that is not topologically conjugate to its inverse
$(X,\phi^{-1})$
(cf. \cite[Example 7.4.19]{LM}),
there is an exact difference between topological conjugacy and flip conjugacy
in Smale spaces.
One of the most interesting problem related to asymptotic continuous orbit equivalence
 is whether or not
there is actually a difference between flip conjugacy and asymptotic continuous orbit equivalence. 
A homeomorphism $h:X\longrightarrow Y$ is said to be 
periodic point preserving if
$h(x)$ is a periodic point in $Y$ for any periodic point $x \in X.$
By studying 
asymptotic periodic points, 
we know that 
if there exists a periodic point preserving homeomorphism
 $h:X\longrightarrow Y$ that gives rise to an asymptotic continuous orbit equivalence between 
irreducible Smale spaces $(X,\phi)$ and $(Y,\psi),$
then they are orbit equivalent with continuous orbit coycle.
Thanks to Boyle-Tomiyama's theorem \cite{BT},
we obtain that $(X,\phi)$ and $(Y,\psi)$ are flip conjugate.
 As a result, we know the following theorem.
\begin{theorem}[{Theorem \ref{thm:flipacoe}}]
Two irreducible Smale spaces  $(X,\phi)$ and $(Y,\psi)$ are flip  conjugate
if and only if there exists a periodic point preserving homeomorphism
 $h:X\longrightarrow Y$ 
that gives rise to an asymptotic continuous orbit equivalence between $(X,\phi)$ and $(Y,\psi).$
\end{theorem}
There is no known examples of Smale spaces that are 
 asymptotically continuous orbit equivalent  but not flip conjugate.

In \cite{MaCJM},
an asymptotic version of  topological conjugacy, called asymptotic conjugacy, 
 was introduced
and characterized in terms of the associated \'etale groupoids $G_\phi^a\rtimes\Z$ 
and its Ruelle algebras $\R_\phi^a$
with the dual action 
$\rho^\phi_t, t \in \T$
induced by $\Z$-crossed product $\R_\phi^a = C^*(G_\phi^a) \rtimes \Z.$ 
The characterizations tell us that 
for two Smale spaces 
$(X,\phi)$ and 
$(Y,\psi)$,   
the following three conditions are equivalent (\cite[Theorem 6.4]{MaCJM}):
\begin{enumerate}
\renewcommand{\theenumi}{\roman{enumi}}
\renewcommand{\labelenumi}{\textup{(\theenumi)}}
\item
$(X,\phi)$ and 
$(Y,\psi)$ are 
asymptotically conjugate.
\item
There exists an isomorphism
$\varphi: G_\phi^a\rtimes\Z\longrightarrow G_\psi^a\rtimes\Z$
 of \'etale groupoids such that 
$d_\psi \circ \varphi = d_\phi,$ 
where
$d_\phi:  G_\phi^a\rtimes\Z\longrightarrow \Z$
is defined by $d_\phi(x,n,z)=n.$
\item
There exists an isomorphism 
$\Phi:\R_\phi^a\longrightarrow\R_\psi^a$
of $C^*$-algebras such that 
$\Phi(C(X) )= C(Y)
$
and
$\Phi \circ \rho^\phi_t =\rho^\psi_{t}\circ \Phi$ for $t \in \T,$
\end{enumerate}
The definition of "asymptotically conjugate" in (i) above  
was given in a complicated fashion in 
\cite[Section 6]{MaCJM}.
We will introduce another notion called {\it asymptotically topologically conjugate}\/
in Smale spaces in the following way.
\begin{definition}
Two Smale spaces $(X,\phi)$ and 
$(Y,\psi)$ 
are said to be {\it asymptotically topologically conjugate}\/
if there exists a homeomorphism $h:X\longrightarrow Y$ satisfying the following two conditions:

\medskip

(A):  $(h(x), h(z)) \in G_\psi^a$ if and only if $(x,z) \in G_\phi^a$,

and the map $\eta_1:(x,z) \in G_\phi^a\longrightarrow (h(x), h(z)) \in G_\psi^a$
is a homeomorphism.

\medskip

(B):  $(\psi(h(x)), h(\phi(x))) \in G_\psi^a$ for all $x \in X$, 

and the map $\xi_1:x \in X \longrightarrow (\psi(h(x)), h(\phi(x))) \in G_\psi^a$ is continuous.
\end{definition} 

We will then prove that 
$(X,\phi)$ and 
$(Y,\psi)$ are asymptotically topologically conjugate if and only if 
they are asymptotically conjugate (Proposition \ref{prop:acatc}), 
so that we obtain the following theorem.

\begin{theorem}[{Theorem \ref{thm:acatc}}]\label{thm:main2}
Let $(X,\phi)$ and 
$(Y,\psi)$ are irreducible Smale spaces. 
The following conditions are equivalent:
\begin{enumerate}
\renewcommand{\theenumi}{\roman{enumi}}
\renewcommand{\labelenumi}{\textup{(\theenumi)}}
\item
$(X,\phi)$ and 
$(Y,\psi)$ are 
asymptotically topologically conjugate.
\item
There exists an isomorphism 
$\Phi:\R_\phi^a\longrightarrow\R_\psi^a$
of $C^*$-algebras such that 
$\Phi(C(X) )= C(Y)
$
and
$\Phi \circ \rho^\phi_t =\rho^\psi_{t}\circ \Phi$ for $t \in \T.$
\end{enumerate}
\end{theorem}

Since the asymptotic continuous orbit equivalence class of a Smale space 
is closed under flip conjugacy, it seems to be natural to introduce a notion 
of asymptotic version of  flip conjugacy. 
 We will then introduce the notion of asymptotic flip conjugacy,
 and  show the following theorem.  
\begin{theorem}[cf. {\cite[Lemma 6.2 and Theorem 6.4]{MaCJM}}]\label{thm:asmpflip}
Let
$(X,\phi)$ and $(Y,\psi)$ be irreducible Smale spaces. 
Then the following assertions are equivalent for $\varepsilon =\pm 1:$
\begin{enumerate}
\renewcommand{\theenumi}{\roman{enumi}}
\renewcommand{\labelenumi}{\textup{(\theenumi)}}
\item
$(X,\phi)$ and 
$(Y,\psi)$ are 
asymptotically  flip conjugate.
\item
There exists an isomorphism
$\varphi: G_\phi^a\rtimes\Z\longrightarrow G_\psi^a\rtimes\Z$
 of \'etale groupoids such that 
$d_\psi \circ \varphi = \varepsilon d_\phi.$
\item
There exists an isomorphism
$\varphi: G_\phi^a\rtimes\Z\longrightarrow G_\psi^a\rtimes\Z$
 of \'etale groupoids such that 
$\varphi(G_\phi^a) = G_\psi^a.$
\item
There exists an isomorphism 
$\Phi:\R_\phi^a\longrightarrow\R_\psi^a$
of $C^*$-algebras such that 
$\Phi(C(X)) = C(Y)
$
and
$\Phi \circ \rho^\phi_t =\rho^\psi_{\varepsilon t}\circ \Phi$ for $t \in \T.$
\item
There exists an isomorphism 
$\Phi:\R_\phi^a\longrightarrow\R_\psi^a$
of $C^*$-algebras such that 
$\Phi(C(X)) = C(Y)
$
and
$\Phi(C^*(G_\phi^a)) = C^*(G_\psi^a).$

%
%
%
\end{enumerate}  
\end{theorem}

If $\dim X =0,$
then the Smale space $(X,\phi)$ must be a shift of finite type,
so called an SFT, and it is topologically conjugate to  
a two-sided topological Markov shift $(\bar{X}_A,\bar{\sigma}_A)$
defined by a square matrix $A$ with entries in $\{0,1\}.$
Let us denote by 
$G_A^a$ and $\R_A^a$ the groupoid
$G_{{\bar\sigma}_A}^a$ and the $C^*$-algebra
$\R_{{\bar\sigma}_A}^a (=C^*(G_{{\bar\sigma}_A}^a\rtimes\Z)),$
 respectively.
The dual action
$\rho^{{\bar\sigma}_A}$ on
$\R_{A}^a$ is denoted by
$\rho^A$.
If we restrict our interest to the class of irreducible topological Markov shifts,
we know that being asymptotically flip conjugate is equivalent to being flip conjugate
(Proposition \ref{prop:flipSFT}), so that we have
\begin{corollary}[{Corollary \ref{cor:flipSFT}}]
Let $A, B$ be irreducible non-permutation matrices with entries in $\{0,1\}.$
Then the following conditions are equivalent: 
\begin{enumerate}
\renewcommand{\theenumi}{\roman{enumi}}
\renewcommand{\labelenumi}{\textup{(\theenumi)}}
\item
The two-sided topological Markov shifts 
$(\bar{X}_A, {\bar\sigma}_A)$ and $ (\bar{X}_B, {\bar\sigma}_B)$
are  flip conjugate.
\item
There exists an isomorphism 
$\Phi:\R_A^a\longrightarrow\R_B^a$
of $C^*$-algebras such that 
$\Phi(C(\bar{X}_A)) = C(\bar{X}_B)
$
and 
$\Phi \circ \rho^A_t =\rho^B_{\varepsilon t}\circ \Phi$ for $t \in \T,$
where 
$\varepsilon = 1$ or $-1.$ 
\item
There exists an isomorphism 
$\Phi:\R_A^a\longrightarrow\R_B^a$
of $C^*$-algebras such that 
$\Phi(C(\bar{X}_A)) = C(\bar{X}_B)
$
and
$\Phi(C^*(G_A^a)) = C^*(G_B^a).$
\end{enumerate}
\end{corollary}
Hence the triplet
$(\R_A^a, C^*(G_A^a), C(\bar{X}_A))$
of $C^*$-subalgebras of $\R_A^a$ 
is a complete invariant of the flip conjugacy class of the two-sided topological Markov shift
$(\bar{X}_A, {\bar\sigma}_A)$.  
The asymptotic flip conjugacy is equivalent to the flip conjugacy for irreducible two-sided topological Markov shifts,
however, for general Smale spaces
it is an open question whether or not  
 asymptotic flip conjugacy implies flip conjugacy.

Throughout the paper, we denote by $\N$ and $\Zp$ 
the set of positive integers and the set of nonnegative integers, respectively.

\section{ Smale spaces and their groupoids}
Let $\phi$ be a homeomorphism on a compact metric space $X$ with metric $d$.
We will briefly recall the definition of Smale space $(X, \phi)$ 
given in \cite[Section 7]{Ruelle1} and \cite{Putnam1}.
The following notations are slightly different from Putnam's ones
\cite{Putnam1}.
 For $\epsilon>0$, 
we put
\begin{equation*}
\Delta_{\epsilon} := \{(x,y) \in X\times X\mid d(x,y) <\epsilon \}.
\end{equation*}
Since 
$(X, \phi)$ 
is a Smale space,
there exists $\epsilon_X>0$ and a continuous map
\begin{equation*}
[\cdot, \cdot]: (x,y) \in \Delta_{\epsilon_X} \longrightarrow [x,y]\in X  
\end{equation*}
which makes $(X, \phi)$ a Smale space
(see \cite{Putnam1}, \cite{Ruelle1}).
Put for $0<\epsilon \le \epsilon_X$
\begin{align}
X^s(x,\epsilon) & = \{ y \in X \mid [y,x] = y, \, d(x,y) <\epsilon \}, \\
X^u(x,\epsilon) & = \{ y \in X \mid [x,y] = y, \, d(x,y) <\epsilon \}.
\end{align}
A Smale space has a hyperbolic structure such as 
 there exists  $0<\lambda_X<1$ such that 
\begin{align}
d(\phi(y), \phi(z)) \le \lambda_X d(y,z) 
\qquad \text{ for } y,z \in X^s(x,\epsilon_X), \label{eq:lambda1}\\
d(\phi^{-1}(y), \phi^{-1}(z)) \le \lambda_X d(y,z) 
\qquad \text{ for } y,z \in X^u(x,\epsilon_X). \label{eq:lambda2}
\end{align}
We call the positive real numbers $\epsilon_X, \lambda_X$
the Smale space constants. 
By Ruelle \cite{Ruelle1} and Putnam \cite{Putnam1},  
there exists $\epsilon_1$ with $0<\epsilon_1<\epsilon_X
$
such that for any $\epsilon$ satisfying
$0<\epsilon <\epsilon_1$, the equalities  
\begin{align}
X^s(x,\epsilon) & = \{ y \in X \mid d(\phi^n(x), \phi^n(y))< \epsilon \text{ for all }
n=0,1,2,\dots \}, \\
X^u(x,\epsilon) & = \{ y \in X \mid d(\phi^n(x), \phi^n(y))< \epsilon \text{ for all }
n=0,-1,-2,\dots \}
\end{align}
hold.
Following Putnam \cite{Putnam1}, we set
\begin{align*}
G_{\phi}^{s,0} &= \{ (x,y) \in X \times X \mid y \in X^s(x,\epsilon_X) \}, \\
G_{\phi}^{u,0} &= \{ (x,y) \in X \times X \mid y \in X^u(x,\epsilon_X) \}, \\
G_{\phi}^{a,0} &= G_{\phi}^{s,0} \cap G_{\phi}^{u,0},
\end{align*}
and for $n\in \Z$, 
\begin{equation*}
G_{\phi}^{s,n} = (\phi\times\phi)^{-n}(G_{\phi}^{s,0}), \quad
G_{\phi}^{u,n} = (\phi\times\phi)^{n}(G_{\phi}^{u,0}), \quad
G_{\phi}^{a,n} = G_{\phi}^{s,n} \cap G_{\phi}^{u,n}.
\end{equation*}
All of them are endowed with the relative topology of $X\times X$.

Since
$\phi(X^s(x,\epsilon)) \subset X^s(\phi(x),\epsilon), $
and
$\phi^{-1}(X^u(x,\epsilon)) \subset X^u(\phi^{-1}(x),\epsilon) $
for any $\epsilon$ with 
$0 < \epsilon \le \epsilon_X$,
 we know that 
\begin{equation}
G_\phi^{*, n} \subset G_\phi^{*, n+1}, \qquad * =s,u,a, \,\, n=0,1,\dots
\label{eq:Gphin}
\end{equation}
  Following \cite{Putnam1}, \cite{Putnam2}, \cite{PutSp},
we define three equivalence relations on $X$:
\begin{equation*}
G_{\phi}^{s} = \cup_{n=0}^\infty G_{\phi}^{s,n},\qquad
G_{\phi}^{u} = \cup_{n=0}^\infty G_{\phi}^{u,n},\qquad
G_{\phi}^{a} = \cup_{n=0}^\infty G_{\phi}^{a,n}.
\end{equation*}
By \eqref{eq:Gphin},
the set $G_{\phi}^{*} = \cup_{n=0}^\infty G_{\phi}^{*,n}$
is an inductive system of topological spaces.
Each $G_\phi^*, * =s,u,a$ is  endowed with the inductive limit topology.
The following lemma has been shown by Putnam. 
\begin{lemma}[Putnam {\cite{Putnam1}}]
\hspace{5cm}
\begin{enumerate}
\renewcommand{\theenumi}{\roman{enumi}}
\renewcommand{\labelenumi}{\textup{(\theenumi)}}
\item
$G_{\phi}^{s} = \{ (x,y) \in X\times X
\mid 
\lim\limits_{n\to\infty}d(\phi^n(x),\phi^n(y)) =0 \}.$
\item
$G_{\phi}^{u} = \{ (x,y) \in X\times X
\mid 
\lim\limits_{n\to\infty}d(\phi^{-n}(x),\phi^{-n}(y)) =0 \}.$
\item
$G_{\phi}^{a} = \{ (x,y) \in X\times X
\mid 
  \lim\limits_{n\to\infty}d(\phi^n(x),\phi^n(y)) 
=\lim\limits_{n\to\infty}d(\phi^{-n}(x),\phi^{-n}(y)) =0 \}.$
\end{enumerate}
\end{lemma}
Putnam has studied these three equivalence relations  
$G_{\phi}^{s}, \, G_{\phi}^{u} $ and $G_{\phi}^{a}$
on $X$ by regarding them as principal groupoids.
In this paper, we will focus on the third equivalence relation 
$G_\phi^a$ on $X$ and its semi-direct product $G_\phi^a\rtimes\Z$ by $\Z$
formulated in \cite{Putnam1} and \cite{PutSp}:
\begin{equation*}
G_{\phi}^{a} \rtimes \Z= \{(x,n,y) \in X \times \Z \times X \mid
(\phi^n(x), y) \in G_{\phi}^{a} \}.
\end{equation*}
Since the map
\begin{equation}
\gamma: 
(x,n,y)\in G_{\phi}^{a} \rtimes \Z \rightarrow 
((x,\phi^{-n}(y)),n)\in G_{\phi}^{a} \times \Z
\label{eq:gamma}
\end{equation}
is bijective, the topology of the groupoid
$G_{\phi}^{a} \rtimes \Z$ is defined by the product topology of 
$G_{\phi}^{a} \times \Z$ through the map $\gamma$.
Let us denote by
$(G_{\phi}^{a} \rtimes \Z)^{(0)}$
the unit space 
\begin{equation*}
(G_{\phi}^{a} \rtimes \Z)^{(0)}
= \{ (x,0,x) \in X \times \Z \times\Z \mid (x,x) \in G_{\phi}^{a}\}
\end{equation*} 
of the groupoid
$G_{\phi}^{a} \rtimes \Z$
which is identified with
that of 
$G_{\phi}^{a}$
and naturally homeomorphic to
the original space $X$
through the correspondence
$(x,0,x) \in (G_{\phi}^{a} \rtimes \Z)^{(0)} \longrightarrow x \in X.$
The range map
$r:G_{\phi}^{a} \rtimes \Z\rightarrow (G_{\phi}^{a} \rtimes \Z)^{(0)}$
and the source map
$s:G_{\phi}^{a} \rtimes \Z\rightarrow (G_{\phi}^{a} \rtimes \Z)^{(0)}$
are defined by 
\begin{equation*}
r(x,n,y) = (x,0,x)
\quad
\text{ and }
\quad
s(x,n,y) = (y,0,y).
\end{equation*}
The groupoid operations are defined by
\begin{gather*}
(x, n, y)\cdot (x',m,w) =(x, n+m, w)\quad \text{ if } y = x',\\
(x, n,y)^{-1} = (y,-n,x).
\end{gather*}

Putnam \cite{Putnam1}, \cite{Putnam2} 
and
Putnam--Spielberg \cite{PutSp}
have  also  studied the groupoid $C^*$-algebras
$C^*(G_\phi^a)$
and
$C^*(G_{\phi}^{a}\rtimes \Z).$
The latter $C^*$-algebra is called the (asymptotic) Ruelle algebra
written $\R_a.$
In this paper we denote it 
by  $\R_\phi^a$
to emphasize the homeomorphism
$\phi$.


\section{Asymptotic continuous orbit equivalence} 
Let $(X,\phi)$ be a Smale space.
In this section, the symbol $d$ will be used as a two-cocycle function
unless we specify, and it does not mean the metric on $X$. 
A one-cocycle function on  $(X,\phi)$ 
means a sequence $\{ f_n \}_{n \in \Z}$ of integer-valued continuous functions on $X$ 
satisfying the identity   
\begin{equation}
f_n(x) + f_m(\phi^n(x)) = f_{n+m}(x),\qquad x \in X, \, n,m \in \Z. \label{eq:onecocycle}
\end{equation}
For a continuous function 
$f:X\longrightarrow \Z$ and $n \in \Z$,
we define a continuous function $f^n$ by setting
\begin{equation}
f^n(x) =
\begin{cases}
\sum_{i=0}^{n-1} f(\phi^i(x)) & \text{ for } n>0, \\
0 & \text{ for } n=0,\\
- \sum_{i=n}^{-1}f(\phi^{i}(x)) & \text{ for } n<0.
\end{cases}    \label{eq:fn}                       
\end{equation}
It is direct to see that the sequence
$\{f^n\}_{n \in \Z}$ 
is a one-cocycle for $\phi$.
Conversely,  one-cocycle function $\{ f_n \}_{n \in \Z}$
is determined by only $f_1$.
Let $G_\phi^a$ be the \'etale groupoid of asymptotic equivalence relation on $X$.
A two-cocycle function on  $(X,\phi)$ 
means a continuous function $d:G_\phi^a \longrightarrow \Z$
satisfying the identity   
\begin{equation}
d(x,z) + d(z,w) = d(x,w), \qquad (x,z), (z,w) \in G_\phi^a. \label{eq:2cocycle}
\end{equation}
The identity \eqref{eq:2cocycle} means that $d: G_\phi^a \longrightarrow \Z$
is a groupoid homomorphism.
In \cite{MaCJM}, the author introduced a notion of 
asymptotic continuous orbit equivalence in Smale spaces in the following way.
\begin{definition}[{\cite[Definition 2.1]{MaCJM}}]\label{def:acoe}
Two Smale spaces
$(X,\phi)$ and $(Y, \psi)$ are said to be
{\it asymptotically continuous orbit equivalent,}\/
written 
$(X, \phi) \underset{ACOE}{\sim}(Y, \psi),$
if there exist a homeomorphism
$h: X\longrightarrow Y$,
continuous functions
$
c_1:X\longrightarrow \Z, \, c_2:Y\longrightarrow \Z,
$
and two-cocycle functions
$
d_1: G_\phi^a \longrightarrow \Z,\, d_2: G_\psi^a \longrightarrow \Z
$
such that 
\begin{enumerate}
\renewcommand{\theenumi}{\arabic{enumi}}
\renewcommand{\labelenumi}{\textup{(\theenumi)}}
\item
$c_1^m(x) + d_1(\phi^m(x), \phi^m(z))
=c_1^m(z) + d_1(x, z), \quad
(x,z) \in G_\phi^a, \, m \in \Z.$
\item
$c_2^m(y) + d_2(\psi^m(y), \psi^m(w))
=c_2^m(w) + d_2(y, w),\quad
(y, w) \in G_\psi^a, \, m \in \Z.$
\end{enumerate}
and,
\begin{enumerate}
\renewcommand{\theenumi}{\roman{enumi}}
\renewcommand{\labelenumi}{\textup{(\theenumi)}}
\item
For each $n \in \Z$, the pair
$( \psi^{c_1^n(x)}(h(x)),  h(\phi^n(x))) =:\xi_1^n(x)$
 belongs to $G_\psi^a$ for each $x \in X$,
and
the map
$\xi_1^n: x \in X \longrightarrow \xi_1^n(x) \in G_\psi^a$ is continuous. 
\item
For each $n \in \Z$, the pair
$( \phi^{c_2^n(y)}(h^{-1}(y)),  h^{-1}(\psi^n(y))) =:\xi_2^n(y)$
 belongs to $G_\phi^a$ for each $y \in Y$,
and
the map
$\xi_2^n: y \in Y \longrightarrow \xi_2^n(y) \in G_\phi^a$ is continuous. 
\item
The pair 
$(\psi^{d_1(x,z)}(h(x)), h(z)) =:\eta_1(x,z)$ belongs to $G_\psi^a$
for each $(x,z) \in G_\phi^a$, and the map
$\eta_1:(x,z) \in G_\phi^a\longrightarrow \eta_1(x,z)\in G_\psi^a$ is continuous.
\item
The pair 
$(\phi^{d_2(y,w)}(h^{-1}(y)), h^{-1}(w)) =:\eta_2(y,w)$ belongs to $G_\phi^a$
for each $(y,w) \in G_\psi^a$, and the map
$\eta_2:(y,w) \in G_\psi^a\longrightarrow \eta_2(y,w)\in G_\phi^a$ is continuous.
\item
$c^{c^n_1(x)}_2(h(x)) + d_2(\psi^{c_1^n(x)}(h(x)), h(\phi^n(x))) = n, \quad x \in X,\, n \in \Z.$
\item
$c^{c^n_2(y)}_1(h^{-1}(y)) + d_1(\phi^{c_2^n(y)}(h^{-1}(y)), h^{-1}(\psi^n(y))) = n, 
\qquad y \in Y,\, n \in \Z.$
\item
$c^{d_1(x,z)}_2(h(x)) + d_2(\psi^{d_1(x,z)}(h(x)), h(z)) = 0, \quad (x,z) \in G_\phi^{a}.$
\item
$c^{d_2(y,w)}_1(h^{-1}(y)) + d_1(\phi^{d_2(y,w)}(h^{-1}(y)), h^{-1}(w)) = 0, 
\qquad (y,w) \in G_\psi^{a}.$
\end{enumerate}
\end{definition}
The listed conditions above of the definition of asymptotic continuous orbit equivalence come from the conditions such that there exists an isomorphism
between $G_\phi^a\rtimes \Z$ and $G_\psi^a\rtimes\Z$ as \'etale groupoids in the following way.
A Smale space $(X,\phi)$ is said to be irreducible
if for every ordered pair of open sets $U, V \subset X$,
there exists $K \in \N$ such that $\phi^K(U) \cap V \ne \emptyset$.  
\begin{proposition}[{\cite[Theorem 3.4]{MaCJM}}]\label{prop:3.2}
Let $(X,\phi)$ and $(Y,\psi)$ be irreducible Smale spaces.  
Then 
$(X, \phi)$ and $(Y, \psi)$ are asymptotically continuous orbit equivalent
if and only if the groupoids 
$G_{\phi}^{a} \rtimes \Z$ and $G_{\psi}^{a} \rtimes \Z$
are isomorphic as \'etale groupoids.
\end{proposition}

In this section, we will reformulate the conditions of Definition \ref{def:acoe}
above in a slightly useful form.
For a homeomorphism $h:X\longrightarrow Y$,
a continuous function $c_1: X\longrightarrow \Z$ and 
a two-cocycle function $d_1: G_\phi^a\longrightarrow\Z,$ 
we define 
\begin{equation*}
\varphi_h(x,n,z) =(h(x),  c_1^n(x) + d_1(\phi^n(x),z), h(z)), 
\qquad (x,n,z) \in G_\phi^a\rtimes\Z. 
\end{equation*}
The map $\varphi_h: G_\phi^a\rtimes\Z\longrightarrow Y\times \Z\times Y$
was used in the proof of \cite[Theorem 3.4]{MaCJM} to define an isomorphism
from $G_\phi^a\rtimes\Z$ to $G_\psi^a\rtimes\Z.$
The following lemma was used in the proof of \cite[Theorem 3.4]{MaCJM} 
without detail proof.
The third condition (iii) below is nothing but the first condition (1) in Definition \ref{def:acoe}.
We will give a detail proof of the lemma here.
\begin{lemma}\label{lem:five}
The following five conditions are mutually equivalent:
\begin{enumerate}
\renewcommand{\theenumi}{\roman{enumi}}
\renewcommand{\labelenumi}{\textup{(\theenumi)}}
\item
$\varphi_h((x,n,x')(x',m,z))
=\varphi_h(x,n,x')\varphi_h(x',m,z)
$ 
for $(x,n,x'), (x',m,z) \in G_\phi^a\rtimes\Z.$ 
\item
$c_1^m(\phi^n(x)) + d_1(\phi^{m+n}(x), \phi^m(z))
=c_1^m(z) + d_1(\phi^n(x), z)
$
for
$
(\phi^n(x),z) \in G_\phi^a, \, m \in \Z.$
\item
$c_1^m(x) + d_1(\phi^m(x), \phi^m(z))
=c_1^m(z) + d_1(x, z)
$
for
$
(x,z) \in G_\phi^a, \, m \in \Z.$
\item
$
c_1(x) + d_1(\phi(x), \phi(z))
=c_1(z) + d_1(x, z)
$
for
$
(x,z) \in G_\phi^a.
$
\item
$
c_1^{-1}(x) + d_1(\phi^{-1}(x), \phi^{-1}(z))
=c_1^{-1}(z) + d_1(x, z)
$
for
$
(x,z) \in G_\phi^a.
$
\end{enumerate}
\end{lemma}
\begin{proof}
(i) $\Longleftrightarrow$ (ii):
For $(x,n,x'), (x',m,z) \in G_\phi^a\rtimes\Z,$
we have
\begin{align*}
& \varphi_h((x,n,x')(x',m,z)) \\
=& \varphi_h(x,n +m,z) \\
=&(h(x), c_1^{n+m}(x) + d_1(\phi^{n+m}(x), z), h(z))\\
=&(h(x), c_1^{n}(x) + c_1^m(\phi^n(x)) + d_1(\phi^{n+m}(x), \phi^m(x'))+ d_1(\phi^{m}(x'), z), h(z))\\
\intertext{and}
&\varphi_h(x,n,x') \varphi_h(x',m,z) \\
=& (h(x), c_1^{n}(x) + d_1(\phi^{n}(x), x'), h(x'))
     (h(x'), c_1^{m}(x') + d_1(\phi^{m}(x'), z), h(z)) \\
=& (h(x), c_1^{n}(x) +c_1^{m}(x') + d_1(\phi^{n}(x), x') +d_1(\phi^{m}(x'), z),  h(z)).
\end{align*}
Hence we see that 
$ \varphi_h((x,n,x')(x',m,z)) =\varphi_h(x,n,x') \varphi_h(x',m,z)$ 
if and only if
\begin{equation*}
c_1^m(\phi^n(x)) + d_1(\phi^{n+m}(x), \phi^m(x'))
=c_1^{m}(x') + d_1(\phi^{n}(x), x')  \text{ for  }(\phi^n(x), x') \in G_\phi^a.
\end{equation*}

(ii) $\Longrightarrow$ (iii): Put $n=0$ in (ii), then we have (iii).

(iii) $\Longrightarrow$ (ii): Take $\phi^n(x)$ as $x$ in (iii), then we have (ii).

(iii) $\Longrightarrow$ (iv): Put $m=1$ in (iii),  then we have (iv).

(iii) $\Longrightarrow$ (v): Put $m=-1$ in (iii),  then we have (v).

(iv) $\Longrightarrow$ (iii): 
Assume (iv).
For $(x,z) \in G_\phi^a$, we know that $(\phi^n(x), \phi^n(z)) \in G_\phi^a$
for all $n \in \Z.$
For $m\in \Z$ with $m>0$,
take $(\phi^n(x), \phi^n(z))$ for $n=0,1,\dots,m-1$ as $(x,z)$ in  (iv).
Then we have 
$$
c_1(\phi^n(x)) + d_1(\phi^{n+1}(x), \phi^{n+1}(z))
=c_1(\phi^n(z)) + d_1(\phi^n(x), \phi^n(z)), \qquad n=0,1,\dots,m-1
$$
so that 
$$
\sum_{n=0}^{m-1} \{ c_1(\phi^n(x)) + d_1(\phi^{n+1}(x), \phi^{n+1}(z)) \}
=\sum_{n=0}^{m-1} \{c_1(\phi^n(z)) + d_1(\phi^n(x), \phi^n(z)) \}.
$$
Hence we get
$$
c_1^m(x) + d_1(\phi^m(x), \phi^m(z))
=c_1^m(z) + d_1(x, z)
\quad \text{ for }
(x,z) \in G_\phi^a, \, m \in \Z, m>0.
$$
Take $(\phi^n(x), \phi^n(z))$ for $n=-1,-2,\dots, -m$ as $(x,z)$
in (iv).
Then we have 
$$
c_1(\phi^n(x)) + d_1(\phi^{n+1}(x), \phi^{n+1}(z))
=c_1(\phi^n(z)) + d_1(\phi^n(x), \phi^n(z)), \qquad n=-1,-2,\dots,-m
$$
so that 
$$
\sum_{n=-m}^{-1} \{ c_1(\phi^n(x)) + d_1(\phi^{n+1}(x), \phi^{n+1}(z)) \}
=\sum_{n=-m}^{-1} \{c_1(\phi^n(z)) + d_1(\phi^n(x), \phi^n(z)) \}.
$$
Hence we get
$$
-c_1^{-m}(x) + d_1(x, z)
=-c_1^{-m}(z) + d_1(\phi^{-m}(x), \phi^{-m}(z))
\quad \text{ for }
(x,z) \in G_\phi^a, \, m \in \Z, m>0.
$$
Therefore (iii) holds for all $m \in \Z.$

(v) $\Longrightarrow$ (iv): Assume (v).
For $(x,z) \in G_\phi^a$, 
take  $(\phi(x), \phi(z)) \in G_\phi^a$
as  $(x,z)$ in (v).
We then have
$$
-c_1(x) + d_1(x, z)
=-c_1(z) + d_1(\phi(x), \phi(z))
$$
so that we obtain (iv).
\end{proof}
By Lemma \ref{lem:five}, the identity (1) in Definition \ref{def:acoe}
is replaced with 
\begin{equation*}
c_1(x) + d_1(\phi(x), \phi(z))
=c_1(z) + d_1(x, z)
\quad \text{ for }
(x,z) \in G_\phi^a,
\end{equation*}
and similarly 
the identity (2) in Definition \ref{def:acoe}
is replaced with 
\begin{equation*}
c_2(y) + d_2(\psi(y), \psi(w))
=c_2(w) + d_2(y, w)
\quad \text{ for }
(y,w) \in G_\psi^a.
\end{equation*}
Let us next reformulate the conditions (i) and (ii) in Definition \ref{def:acoe}.
We provide lemmas.
\begin{lemma}\label{lem:3.4-1}
Let $\lambda_X$ be the positive constant less than one appearing in \eqref{eq:lambda1}
and \eqref{eq:lambda2}.
Let $m,n \in \Z$ satisfy $m\cdot n >0$ and 
$\lambda_X^{|m|} + \lambda_X^{|n|} < 1.$
Then we have
\begin{enumerate}
\renewcommand{\theenumi}{\roman{enumi}}
\renewcommand{\labelenumi}{\textup{(\theenumi)}}
\item
For $(x,y) \in G_\phi^{s,n},\, (y,z) \in G_\phi^{s,m},$ we have
\begin{equation*}
(x,z) \text{ belongs to } G_\phi^{s,0} \text{  if } n,m<0, \text{  and }  
(x,z) \text{ belongs to } G_\phi^{s,n+m} \text{  if } n,m>0.
\end{equation*}
\item
For $(x,y) \in G_\phi^{u,n},\, (y,z) \in G_\phi^{u,m},$ we have
\begin{equation*}
(x,z) \text{ belongs to } G_\phi^{u,0} \text{  if } n,m<0, \text{  and }  
(x,z) \text{ belongs to } G_\phi^{u,n+m} \text{  if } n,m>0.
\end{equation*}
\end{enumerate}
\end{lemma}
\begin{proof}
(i) We first assume that  $n, m <0.$
Put $k =-n, \, l= -m \in \N.$
For
$(x,y) \in G_\phi^{s,n},$ 
put
$x'= \phi^n(x), y' = \phi^n(y)
$
so that 
$(x', y') \in G_\phi^{s,0}.$
We then have
$$
d(x,y) =d(\phi^k(x'), \phi^k(y')) 
\le 
\lambda_X^k d(x', y') 
< 
\lambda_X^k \cdot \epsilon_X
$$
and
$$
[x,y] = [\phi^k(x'),\phi^k(y')] = \phi^k([x',y']) = \phi^k(x') = x.
$$
Similarly for $(y,z) \in G_\phi^{s,m},$
we have
$
d(y, z) < \lambda_X^l\cdot \epsilon_X
$ 
and
$
[y,z] = y,
$
so that we know
\begin{gather*}
d(x,z) \le d(x,y) + d(y,z) < (\lambda_X^k + \lambda_X^l) \cdot \epsilon_X <\epsilon_X, \\
[x,z] = [[x, y], z]] = [x,[y,z]] = [x,y] = x,
\end{gather*}
proving  $(x,z) \in G_\phi^{s,0}.$

We next assume that 
$n, m >0.$
For
$(x,y) \in G_\phi^{s,n}$ 
and
$(y,z) \in G_\phi^{s,m},$
so that we see
$(\phi^n(x), \phi^n(y)),(\phi^m(y), \phi^m(z)) \in G_\phi^{s,0}.$
Hence we have
$$
d(\phi^{n+m}(x),\phi^{n+m}(y)) < \lambda_X^m \cdot \epsilon_X, \qquad
d(\phi^{n+m}(y),\phi^{n+m}(z)) < \lambda_X^n\cdot \epsilon_X
$$
so that 
\begin{equation*}
d(\phi^{n+m}(x),\phi^{n+m}(z)) 
<   \lambda_X^m\cdot \epsilon_X + \lambda_X^n \cdot \epsilon_X 
<\epsilon_X.
\end{equation*}
We also have 
\begin{align*}
[\phi^{n+m}(x),\phi^{n+m}(z)]
& =  [[ \phi^{n+m}(x),\phi^{n+m}(y)],\phi^{n+m}(z)] \\
& =  [\phi^{n+m}(x),[\phi^{n+m}(y),\phi^{n+m}(z)]] \\
& =[\phi^{n+m}(x),\phi^{n+m}(y)] \\
& =\phi^m([\phi^{n}(x),\phi^{n}(y)]) =\phi^{n+m}(x),
\end{align*} 
proving  $(\phi^{n+m}(x),\phi^{n+m}(z)) \in G_\phi^{s,n+m}.$

(ii) is similarly shown to (i).
\end{proof}
\begin{lemma}\label{lem:3.5.1}
Let $(X,\phi)$ and $(Y,\psi)$ be Smale spaces.
Suppose that there exist a homeomorphism
 $h: X \longrightarrow Y,$
 a continuous function $c_1:X \longrightarrow \Z$
and a nonnegative integer $K\in\Zp$ such that 
\begin{equation}
 (\psi^{c_1(x)}(h(x)), h(\phi(x)))\in G_\psi^{a,K}, \qquad x \in X.  \label{eq:3.5.1.1}
\end{equation}
Then there exists a nonnegative integer $K_n \in \Zp$ for each $n \in \Z$ such that 
\begin{equation}
 (\psi^{c_1^n(x)}(h(x)), h(\phi^n(x)))\in G_\psi^{a,K_n}, \qquad x \in X.  \label{eq:3.5.1.11}
\end{equation}
\end{lemma}
\begin{proof}
Put $C_1 = \Max\{|c_1(x) | \mid x \in X\}.$
Assume that there exists 
a nonnegative integer $K_m \in \Zp$ for a fixed $n=m \in \Zp$ such that 
\begin{equation}
 (\psi^{c_1^m(x)}(h(x)), h(\phi^m(x)))\in G_\psi^{a,K_m}, \qquad x \in X.  \label{eq:3.5.1.2}
\end{equation}
Since $G_\psi^{a, n} \subset G_\psi^{a, n+1}$ for any $n \in \N,$
one may assume that the nonnegative integers $K$ in  \eqref{eq:3.5.1.1}  
and $K_m$ in \eqref{eq:3.5.1.2} satisfy  
$
0<\lambda_X^K < \frac{1}{2}
$
and
$
K_m + C_1 \ge K.
$
We then have
\begin{align*}
& (\psi^{c_1^{m+1}(x)}(h(x)), h(\phi^{m+1}(x))) \\
= & (\psi\times\psi)^{c_1(\phi^m(x))}(\psi^{c_1^{m}(x)}(h(x)), h(\phi^{m}(x))) \cdot
    (\psi^{c_1(\phi^m(x))}(h(\phi^m(x))), h(\phi(\phi^{m}(x))) ).
\end{align*}
By \eqref{eq:3.5.1.2}, we have 
\begin{equation*}
(\psi\times\psi)^{c_1(\phi^m(x))}(\psi^{c_1^{m}(x)}(h(x)), h(\phi^{m}(x)))
\in G_\psi^{s,K_m -c_1(\phi^m(x))} \subset G_\psi^{s,K_m + C_1}.
\end{equation*}
Together with the assumption \eqref{eq:3.5.1.1},
Lemma \ref{lem:3.4-1} (i) tells us that 
\begin{equation*}
(\psi^{c_1^{m+1}(x)}(h(x)), h(\phi^{m+1}(x))) \in G_\psi^{s,K_m + C_1 + K}.
\end{equation*}
By induction,  
there exists a nonnegative integer  $K_n \in \Zp$ such that   
\begin{equation}
(\psi^{c_1^{n}(x)}(h(x)), h(\phi^{n}(x))) \in G_\psi^{s,K_n} \quad \text{ for all } n \in \N.
\label{eq:3.5.1.4}
\end{equation}
Similarly one may show the above \eqref{eq:3.5.1.4}
for the unstable set $G_\psi^{u, K_n}$,
so that we obtain 
\eqref{eq:3.5.1.11} for positive integers $n \in \N.$

For negative integers $n \in \Z,$
by taking $\phi^{-1}(x)$ as $x$ in \eqref{eq:3.5.1.1}, 
one may show that there exists $K' \in \Zp$ such that 
\begin{equation*}
(\psi^{c_1^{-1}(x)}(h(x)), h(\phi^{-1}(x))) \in G_\psi^{a,K'},  \qquad x \in X.
\end{equation*}
By a similar manner to the above discussion showing \eqref{eq:3.5.1.11} 
for positive integers $n \in \N,$
one may find a nonnegative integer $K_{-n}$ for $n \in \N$ such that  
\begin{equation*}
(\psi^{c_1^{-n}(x)}(h(x)), h(\phi^{-n}(x))) \in G_\psi^{a,K_{-n}} \quad \text{ for all } n \in \N
\end{equation*}
by using Lemma \ref{lem:3.4-1},
so that the desired assertion is verified.
\end{proof}
 Hence we have the following lemma.
 
\begin{lemma}\label{lem:3.3}
Let $(X,\phi)$ and $(Y,\psi)$ be Smale spaces.
Suppose that there exist a homeomorphism
 $h: X \longrightarrow Y$
 and  continuous functions 
 $c_1:X \longrightarrow \Z,$
 $k_0:X \longrightarrow \Zp$
 such that 
 \begin{equation}
(\psi^{c_1(x)}(h(x)), h(\phi(x)))\in G_\psi^{a,k_0(x)}, \qquad x \in X. 
\label{eq:3.5.2.1}
\end{equation}
Then there exists a sequence of continuous functions
$k_{1,n}: X \longrightarrow \Zp$ for $n \in \Z$ such that 
\begin{equation}
\xi_1^n(x):= ( \psi^{c_1^n(x)}(h(x)),  h(\phi^n(x))) \in  G_\psi^{a,k_{1,n}(x)}
\text{ for } x \in X \text{  and } n \in \Z.
\end{equation}
Hence $\xi_1^n(x)$ belongs to $G_\psi^a$ and
$\xi_1^n: X \longrightarrow G_\psi^a$
is continuous.
\end{lemma}
\begin{proof}
Put the nonnegative integer 
 $K = \Max \{ k_0(x) \mid x \in X \}.$
 Then the condition  
 \eqref{eq:3.5.2.1} implies 
  \begin{equation*}
(\psi^{c_1(x)}(h(x)), h(\phi(x)))\in G_\psi^{a,K}, \qquad x \in X. 
\end{equation*}
 Hence by Lemma \ref{lem:3.5.1},
 one may take the desired continuous function
 $k_{1,n}(x)$ as a constant integer $K_n$ for each $n \in \Z.$ 
\end{proof}

We provide one more lemma.

\begin{lemma}
Both the condition (2) for $n=1$ in Definition \ref{def:acoe}:
\begin{equation}
c_2(y) + d_2(\psi(y), \psi(w)) = c_2(w) + d_2(y,w), \qquad (y,w) \in G_\psi^a 
\label{eq:2.4.1} 
\end{equation}
and
the condition (v) for $n=1$ in Definition \ref{def:acoe}:
\begin{equation}
c_2^{c_1(x)}(h(x)) + d_2(\psi^{c_1(x)}(h(x)), h(\phi(x))) = 1, \qquad x \in X 
\label{eq:2.4.2} 
\end{equation}
imply 
the condition (v) for $n\in \Z$ in Definition \ref{def:acoe}:
\begin{equation}
c_2^{c_1^n(x)}(h(x)) + d_2(\psi^{c_1^n(x)}(h(x)), h(\phi^n(x))) = n, \qquad x \in X,\, n \in \Z. 
\label{eq:vn} 
\end{equation}
Similarly, 
both the condition (1) for $n=1$ in Definition \ref{def:acoe}
and
the condition (vi) for $n=1$ in Definition \ref{def:acoe}
imply 
the condition (vi) for $n\in \Z$ in Definition \ref{def:acoe}.
\end{lemma}
\begin{proof}
We will first prove the equality \eqref{eq:vn} for positive integers $n$ by induction on $n \in \N.$ 
By Lemma \ref{lem:five}, the equality \eqref{eq:2.4.1} is equivalent to the equality 
\begin{equation}
c_2^m(y) + d_2(\psi^m(y), \psi^m(w)) = c_2^m(w) + d_2(y,w), \qquad (y,w) \in G_\psi^a, m \in \Z.
\label{eq:2.4.4} 
\end{equation}
Assume the equality \eqref{eq:vn} for a fixed $n=m \in \N$ and all $x \in X.$
Take $\phi(x)$ for $x$ in \eqref{eq:vn} for a fixed $m \in \N$, we have
\begin{equation}
c_2^{c_1^m(\phi(x))}(h(\phi(x))) + d_2(\psi^{c_1^m(\phi(x))}(h(\phi(x))), h(\phi^{m+1}(x))) = m, \qquad x \in X. 
\label{eq:maru1} 
\end{equation}
Take $c_1^m(\phi(x)), \psi^{c_1(x)}(h(x)), h(\phi(x))$ as $m, y, w,$ respectively 
 in \eqref{eq:2.4.4},
so that we have 
\begin{align}
& c_2^{c_1^m(\phi(x))}(\psi^{c_1(x)}(h(x))) 
+ d_2(\psi^{c_1^m(\phi(x))}(\psi^{c_1(x)}(h(x))),
\psi^{c_1^m(\phi(x))}(h(\phi(x)))) \\
 =&  c_2^{c_1^m(\phi(x))}(h(\phi(x))) + d_2(\psi^{c_1(x)}(h(x)), h(\phi(x))). 
\label{eq:2.4.6} 
\end{align}
By  \eqref{eq:2.4.6} with \eqref{eq:2.4.2},
we have
\begin{align}
& c_2^{c_1^m(\phi(x))}(\psi^{c_1(x)}(h(x))) 
+ d_2(\psi^{c_1^m(\phi(x))}(\psi^{c_1(x)}(h(x))),
\psi^{c_1^m(\phi(x))}(h(\phi(x)))) \\
 =&  c_2^{c_1^m(\phi(x))}(h(\phi(x))) + 1 - c_2^{c_1(x)}(h(x)). \label{eq:maru2} 
\end{align}
By \eqref{eq:maru1} and \eqref{eq:maru2}, 
we have
\begin{equation*}
 c_2^{c_1^m(\phi(x))}(\psi^{c_1(x)}(h(x))) 
+ d_2(\psi^{c_1^m(\phi(x))}(\psi^{c_1(x)}(h(x))),
h(\phi^{m+1}(x))) 
 = m+1 -  c_2^{c_1(x)}(h(x)).  
\end{equation*}
 This shows the equality
\begin{equation*}
c_2^{c_1^{m+1}(x)}(h(x)) + d_2(\psi^{c_1^{m+1}(x)}(h(x)), h(\phi^{m+1}(x))) = m+1, \qquad x \in X,
\end{equation*}
proving the equality \eqref{eq:vn} holds for $n=m+1$ and hence for all $n \in \N.$ 

We will next prove \eqref{eq:vn} for negative integers $n.$
Take $c_1(\phi^{-1}(x)),
 \psi^{-c_1(\phi^{-1}(x))}(h(x)), 
h(\phi^{-1}(x))$
as $m,y, w,$ respectively
in \eqref{eq:2.4.4}, so that 
\begin{align*}
& c_2^{c_1(\phi^{-1}(x))}(\psi^{-c_1(\phi^{-1}(x))}(h(x))) 
+ d_2(h(x), \psi^{c_1(\phi^{-1}(x))}(h(\phi^{-1}(x)))) \\
= & c_2^{c_1(\phi^{-1}(x))}(h(\phi^{-1}(x))) 
+ d_2(\psi^{-c_1(\phi^{-1}(x))}(h(x)),h(\phi^{-1}(x))).
\end{align*}
It then follows that 
\begin{align*}
& c_2^{c_1^{-1}(x)}(h(x)) + d_2(\psi^{c_1^{-1}(x)}(h(x)), h(\phi^{-1}(x))) \\
=
& c_2^{-c_1(\phi^{-1}(x))}(h(x)) + d_2(\psi^{-c_1(\phi^{-1}(x))}(h(x)), h(\phi^{-1}(x))) \\
=
& c_2^{-c_1(\phi^{-1}(x))}(h(x)) +c_2^{c_1(\phi^{-1}(x))}(\psi^{-c_1(\phi^{-1}(x))}(h(x))) \\
& + d_2(h(x), \psi^{c_1(\phi^{-1}(x))}(h(\phi^{-1}(x)))) 
 -c_2^{c_1(\phi^{-1}(x))}(h(\phi^{-1}(x))) \\
=
& - d_2( \psi^{c_1(\phi^{-1}(x))}(h(\phi^{-1}(x))), h(x)) 
 -c_2^{c_1(\phi^{-1}(x))}(h(\phi^{-1}(x))) = -1
\end{align*}
because of the equality \eqref{eq:2.4.2} for $\phi^{-1}(x)$ as $x$.
Hence the equality \eqref{eq:vn} holds for $n= -1.$
Assume the equality \eqref{eq:vn} holds for a fixed $n= -m$ for $m \in \N$
and for all $ x \in X.$
Take $\phi^{-1}(x)$ as $x$ in \eqref{eq:vn}, 
so that 
\begin{equation}
 c_2^{c_1^{-m}(\phi^{-1}(x))}(h(\phi^{-1}(x))) 
+ d_2(\psi^{c_1^{-m}(\phi^{-1}(x))}(h(\phi^{-1}(x))),
 h(\phi^{-(m+1)}(x))) 
=-m.\label{eq:3.41}
\end{equation}
Take $c_1^{-m}(\phi^{-1}(x)),
 \psi^{c_1^{-1}(x)}(h(x)), 
h(\phi^{-1}(x))$
as $m,y,w,$ respectively
in \eqref{eq:2.4.4}, so that we have with \eqref{eq:vn} for $n=-1$
\begin{align*}
& c_2^{c_1^{-m}(\phi^{-1}(x))}(\psi^{c_1^{-1}(x)}(h(x))) 
+ d_2(\psi^{c_1^{-m}(\phi^{-1}(x))}(\psi^{c_1^{-1}(x)}(h(x))), 
\psi^{c_1^{-m}(\phi^{-1}(x))}(h(\phi^{-1}(x)))) \\
= & 
c_2^{c_1^{-m}(\phi^{-1}(x))}(h(\phi^{-1}(x))) 
+ d_2(\psi^{c_1^{-1}(x)}(h(x)), h(\phi^{-1}(x))) \\
= & 
c_2^{c_1^{-m}(\phi^{-1}(x))}(h(\phi^{-1}(x))) 
+ (-1 - c_2^{c_1^{-1}(x)}(h(x))).
\end{align*}
By \eqref{eq:3.41}, we obtain that 
\begin{align*}
& c_2^{c_1^{-m}(\phi^{-1}(x))}(\psi^{c_1^{-1}(x)}(h(x))) 
+ d_2(\psi^{c_1^{-m}(\phi^{-1}(x)) +c_1^{-1}(x)}(h(x)), 
\psi^{c_1^{-m}(\phi^{-1}(x))}(h(\phi^{-1}(x)))) \\
= &
-m  -d_2(\psi^{c_1^{-m}(\phi^{-1}(x))}(h(\phi^{-1}(x))),
 h(\phi^{-(m+1)}(x)))
-1 - c_2^{c_1^{-1}(x)}(h(x)) \\
\intertext{and hence}
& c_2^{c_1^{-m}(\phi^{-1}(x))}(\psi^{c_1^{-1}(x)}(h(x))) 
+c_2^{c_1^{-1}(x)}(h(x)) \\
&+ d_2(\psi^{c_1^{-(m+1)}(x) }(h(x)), 
\psi^{c_1^{-m}(\phi^{-1}(x))}(h(\phi^{-1}(x)))) 
+d_2(\psi^{c_1^{-m}(\phi^{-1}(x))}(h(\phi^{-1}(x))), h(\phi^{-(m+1)}(x))) \\
= & -(m+1).
\end{align*}
We thus have
\begin{equation*}
c_2^{c_1^{-(m+1)}(x)}(h(x)) + d_2(\psi^{c_1^{-(m+1)}(x)}(h(x)), h(\phi^{-(m+1)}(x))) 
= {-(m+1)},
\end{equation*}
completing the induction.

Similarly we may prove the condition (vi) for $n\in \Z$ 
from both
the condition (vi) for $n=1$
and
the condition (1) for $n=1.$
\end{proof}
Therefore we may reformulate  Definition \ref{def:acoe} in the following way.
\begin{proposition}\label{prop:main1}
Two Smale spaces
$(X,\phi)$ and $(Y, \psi)$ are 
 asymptotically continuous orbit equivalent
if and only if there exist a homeomorphism
$h: X\longrightarrow Y$,
continuous functions
$
c_1:X\longrightarrow \Z, \, c_2:Y\longrightarrow \Z,
$
and two-cocycle functions
$
d_1: G_\phi^a \longrightarrow \Z,\, d_2: G_\psi^a \longrightarrow \Z
$
such that 
\begin{enumerate}
\renewcommand{\theenumi}{\arabic{enumi}}
\renewcommand{\labelenumi}{\textup{(\theenumi)}}
\item
$c_1(x) + d_1(\phi(x), \phi(z))
=c_1(z) + d_1(x, z), \quad
(x,z) \in G_\phi^a.$
\item
$c_2(y) + d_2(\psi(y), \psi(w))
=c_2(w) + d_2(y, w),\quad
(y, w) \in G_\psi^a.$
\end{enumerate}
and,
\begin{enumerate}
\renewcommand{\theenumi}{\roman{enumi}}
\renewcommand{\labelenumi}{\textup{(\theenumi)}}
\item
The pair
$( \psi^{c_1(x)}(h(x)),  h(\phi(x))) =:\xi_1(x)$
 belongs to $G_\psi^a$ for each $x \in X$,
and
the map
$\xi_1: x \in X \longrightarrow \xi_1(x) \in G_\psi^a$ is continuous. 
\item
The pair
$( \phi^{c_2(y)}(h^{-1}(y)),  h^{-1}(\psi(y))) =:\xi_2(y)$
 belongs to $G_\phi^a$ for each $y \in Y$,
and
the map
$\xi_2: y \in Y \longrightarrow \xi_2(y) \in G_\phi^a$ is continuous. 
\item
The pair 
$(\psi^{d_1(x,z)}(h(x)), h(z)) =:\eta_1(x,z)$ belongs to $G_\psi^a$
for each $(x,z) \in G_\phi^a$, and the map
$\eta_1:(x,z) \in G_\phi^a\longrightarrow \eta_1(x,z)\in G_\psi^a$ is continuous.
\item
The pair 
$(\phi^{d_2(y,w)}(h^{-1}(y)), h^{-1}(w)) =:\eta_2(y,w)$ belongs to $G_\phi^a$
for each $(y,w) \in G_\psi^a$, and the map
$\eta_2:(y,w) \in G_\psi^a\longrightarrow \eta_2(y,w)\in G_\phi^a$ is continuous.
\item
$c^{c_1(x)}_2(h(x)) + d_2(\psi^{c_1(x)}(h(x)), h(\phi(x))) =1, \quad x \in X.$
\item
$c^{c_2(y)}_1(h^{-1}(y)) + d_1(\phi^{c_2(y)}(h^{-1}(y)), h^{-1}(\psi(y))) =1, 
\qquad y \in Y.$
\item
$c^{d_1(x,z)}_2(h(x)) + d_2(\psi^{d_1(x,z)}(h(x)), h(z)) = 0, \quad (x,z) \in G_\phi^{a}.$
\item
$c^{d_2(y,w)}_1(h^{-1}(y)) + d_1(\phi^{d_2(y,w)}(h^{-1}(y)), h^{-1}(w)) = 0, 
\qquad (y,w) \in G_\psi^{a}.$
\end{enumerate}
\end{proposition}
It is straightforward to see that 
the inverse $(X, \phi^{-1})$ for a Smale space 
$(X,\phi)$ also becomes a Smale space in a natural way. 
Before ending this section,
we notice the following proposition which has been seen in \cite{MaCJM}.
\begin{proposition}[{\cite[Proposition 11.1]{MaCJM}}]
A Smale space $(X,\phi)$ is asymptotically continuous orbit equivalent to its inverse
$(X,\phi^{-1}).$
\end{proposition}
\begin{proof}
We set $Y = X, \psi = \phi^{-1}$
and
$h =\id, c_1\equiv -1, c_2\equiv -1, d_1\equiv 0, d_1\equiv 0$
in Proposition \ref{prop:main1}.
They satisfy all of the conditions in Proposition \ref{prop:main1}
to show
$(X, \phi) \underset{ACOE}{\sim}(X, \phi^{-1}).$ 
\end{proof}

\section{Asymptotic topological conjugacy}

In \cite{MaCJM},
an asymptotic version of  topological conjugacy in Smale spaces, called asymptotic conjugacy, 
 was introduced.
The definition of "asymptotically conjugate"   was given  
in the following way.
\begin{definition}[{\cite[Definition 6.1]{MaCJM}}] 
Two Smale spaces $(X,\phi)$ and 
$(Y,\psi)$ 
are {\it asymptotically conjugate}\/
if they are asymptotically continuous orbit equivalent such that 
one can take their cocycle functions as
$c_1\equiv 1, c_2 \equiv 1$ and $d_1\equiv 0, d_2\equiv 0$
in Definition \ref{def:acoe}. 
\end{definition}
In this section, instead of the above "asymptotically conjugate", 
we will introduce a notion called {\it asymptotically topologically conjugate}\/
in Smale spaces that is more natural and easier to formulate than "asymptotically conjugate".
\begin{definition}\label{def:atc}
Two Smale spaces $(X,\phi)$ and 
$(Y,\psi)$ 
are said to be {\it asymptotically topologically conjugate}\/
if there exists a homeomorphism $h:X\longrightarrow Y$ satisfying the following two conditions:

\medskip

(A):  $(h(x), h(z)) \in G_\psi^a$ if and only if $(x,z) \in G_\phi^a$,

and the map $\eta_1:(x,z) \in G_\phi^a\longrightarrow (h(x), h(z)) \in G_\psi^a$
is a homeomorphism.

\medskip

(B):  $(\psi(h(x)), h(\phi(x))) \in G_\psi^a$ for all $x \in X$, 

and the map 
$\xi_1:x \in X \longrightarrow (\psi(h(x)), h(\phi(x))) \in G_\psi^a$ is continuous.
\end{definition} 

 \begin{lemma}
 Under the conditions (A) and (B), the following condition (B') holds
  
  \medskip
  
 (B'):  $(\phi(h^{-1}(y)), h^{-1}(\psi(y))) \in G_\phi^a$ for all $y \in Y$,

 and the map $\xi_2:y \in Y \longrightarrow (\phi(h^{-1}(y)), h^{-1}(\psi(y))) \in G_\phi^a$
 is continuous. 
 \end{lemma}
 \begin{proof}
 Under the conditions (A) and (B),
 we have for $y \in Y$, putting $ x = h^{-1}(y) \in X$, 
 \begin{equation*}
 (h^{-1}(\psi(y)), \phi(h^{-1}(y))) 
 = \eta_1^{-1}(\psi(y), h(\phi(h^{-1}(y)))) 
 = \eta_1^{-1}(\psi(h(x)), h(\phi(x))).
 \end{equation*}
 By the conditions (A) and (B), we have 
 $\eta_1^{-1}(\psi(h(x)), h(\phi(x))) \in G_\phi^a$ 
so that 
 $(h^{-1}(\psi(y)), \phi(h^{-1}(y)))  \in G_\phi^a$ and hence
 $ (\phi(h^{-1}(y)), h^{-1}(\psi(y))) \in G_\phi^a$.
 Since $\xi_2 = \eta_1^{-1} \circ \xi_1\circ h^{-1}$,
 we know that $\xi_2: Y \longrightarrow G_\phi^a$ is continuous. 
 \end{proof}
 We note the following proposition.
 \begin{proposition}
Let $h: X \longrightarrow Y$ be a homeomorphism
between Smale spaces $(X,\phi) $ and $(Y,\psi)$.
\begin{enumerate}
\renewcommand{\theenumi}{\roman{enumi}}
\renewcommand{\labelenumi}{\textup{(\theenumi)}}
\item 
Assume that $(h(x), h(z))\in G_\psi^a$ for all $(x,z) \in G_\phi^a$.
Then $\eta_1:(x,z) \in G_\phi^a\longrightarrow (h(x), h(z)) \in G_\psi^a$
is continuous if and only if for any $N \in \N$, there exists $N_1\in \N$ 
such that $\eta_1(G_\phi^{a,N}) \subset G_\psi^{a,N_1}$.
\item 
Assume that $(\psi(h(x)), h(\phi(x)))\in G_\psi^a$ for all $x \in X$.
Then $\xi_1: x\in X \longrightarrow (\psi(h(x)), h(\phi(x))) \in G_\psi^a$
is continuous if and only if there exists $N_1\in \N$ 
such that $\xi_1(X) \subset G_\psi^{a,N_1}$.
\end{enumerate} 
 \end{proposition}
 \begin{proof}
 (i)
 Assume that $\eta_1:(x,z) \in G_\phi^a\longrightarrow (h(x), h(z)) \in G_\psi^a$
is continuous.
  Take an arbitrary fixed $N \in \N$. 
 Let $ \epsilon_X, \lambda_X$ 
be the Smale space constants for the Smale space $X$ as in Section 2.
 Since 
$G_\phi^{s,N} = \{ (x,z) \in X\times X \mid d(\phi^n(x), \phi^n(z)) <\epsilon_X, n=N, N+1,\dots \}$,
 we have
 \begin{equation*}
 \overline{G_\phi^{s,N}} \subset 
\{ (x,z) \in X\times X \mid d(\phi^n(x), \phi^n(z)) \le \epsilon_X, n=N, N+1,\dots \}.
 \end{equation*}
 By \eqref{eq:lambda1} with $0<\lambda_X<1$, 
we have for 
$(x,z) \in 
\overline{G_\phi^{s,N}}$
and
$n=N, N+1,\dots $
 \begin{equation*}
d(\phi^{n+1}(x), \phi^{n+1}(z)) 
\le \lambda_X d(\phi^n(x),\phi^n(z))
<  d(\phi^n(x),\phi^n(z))
\le \epsilon_X,
 \end{equation*}
 so that 
 $(x,z) \in G_\phi^{s,N+1}$.
 Hence we have
 $\overline{G_\phi^{s,N}} \subset G_\phi^{s,N+1}$.
 Similarly we have 
  $\overline{G_\phi^{u,N}} \subset G_\phi^{u,N+1}$.
 We thus have
 \begin{equation*}
 G_\phi^{a,N} \subset\overline{G_\phi^{a,N}} \subset G_\phi^{a,N+1}
 \end{equation*}
 and hence
 \begin{equation*}
 \eta_1(G_\phi^{a,N}) \subset \eta_1(\overline{G_\phi^{a,N}}) \subset\eta_1( G_\phi^{a,N+1}).
 \end{equation*}
 As $\eta_1:G_\phi^a\longrightarrow G_\psi^a$ is continuous,
 $\eta_1(\overline{G_\phi^{a,N}})$ is compact in $G_\psi^a$.
 Since
 $G_\psi^a = \cup_{n=0}^\infty G_\psi^{a,n}$ and we may find $N_1\in \N$ such that 
  $\eta_1(\overline{G_\phi^{a,N}})\subset G_\psi^{a,N_1}$ 
and hence
 $\eta_1(G_\phi^{a,N}) \subset G_\psi^{a,N_1}.$

 Conversely, assume that 
  for any $N \in \N$, there exists $N_1\in \N$ 
such that $\eta_1(G_\phi^{a,N}) \subset G_\psi^{a,N_1}$.
For any $(x,z) \in G_\phi^a$,
take $N\in \N$ such that 
$(x,z) \in G_\phi^{a,N}$.
Since
the topology of $G_\phi^{a,N}$ and $G_\psi^{a,N_1}$
are given by the relative topology in $X\times X$ and $Y\times Y$, respectively,
and $\eta_1: G_\phi^{a,N}\longrightarrow G_\psi^{a,N_1}$ 
is the restriction of 
$h\times h : X \times X\longrightarrow Y\times Y$,
we know that 
$\eta_1$ is continuous at $(x,z)$.
Therefore 
$\eta_1:G_\phi^a\longrightarrow G_\psi^a$ is continuous.

 (ii)
 Assume that $\xi_1: x\in X \longrightarrow (\psi(h(x)), h(\phi(x))) \in G_\psi^a$
is continuous.
Since $X$ is compact, so is $\xi_1(X)$ in $G_\psi^a = \cup_{n=0}^\infty G_\psi^{a,n}$.
One may find
$N_1\in \N$ 
such that $\xi_1(X) \subset G_\psi^{a,N_1}$.

 Conversely
 assume that 
there exists $N_1\in \N$ 
such that $\xi_1(X) \subset G_\psi^{a,N_1}$.
As $\xi_1(x) = (\psi(h(x)), h(\phi(x))) \in G_\psi^{a,N_1}$
and
the topology of $G_\psi^{a,N_1}$ 
is given by the relative topology in $Y \times Y$,
we easily know that $\xi_1:X\longrightarrow G_\psi^a$ is continuous. 
 \end{proof}
 
 \begin{lemma}\label{lem:N6.4}
Let $h: X \longrightarrow Y$ be a homeomorphism
between Smale spaces $(X,\phi) $ and $(Y,\psi)$.
Let 
continuous functions $c_1:X\longrightarrow \Z$,
 $c_2:Y\longrightarrow \Z$ 
 and
 two-cocycle functions
 $d_1:G_\phi^a\longrightarrow \Z$, $d_2:G_\psi^a\longrightarrow\Z$ 
be given by 
 $c_1\equiv 1, c_2\equiv 1, d_1\equiv 0, d_2\equiv 0$.
 Then $(h, c_1, c_2, d_1, d_2)$ 
satisfies all of the conditions of Proposition \ref{prop:main1}
 if and only if $h$ 
satisfies the conditions (A) and (B) in Definition \ref{def:atc}.
  \end{lemma}
 \begin{proof}
As  $c_1\equiv 1, c_2\equiv 1, d_1\equiv 0$ and $d_2\equiv 0$,
the conditions (1) and (2) 
 as well as (v), (vi), (vii) and (viii) 
 in Proposition  \ref{prop:main1} always hold.
  Since 
 \begin{gather*}
 \xi_1(x) = (\psi(h(x)), h(\phi(x))) \quad \text{ for } x \in X, \\
 \xi_2(y) = (\phi(h^{-1}(y)), h^{-1}(\psi(y))) \quad \text{ for } y \in Y,
 \end{gather*}
 both the conditions (i) and (ii) hold 
if and only if $h$ satisfies the conditions (B) and (B').
 Since
 \begin{gather*}
 \eta_1(x,z) = (h(x), h(z)) \quad \text{ for } (x,z) \in G_\phi^a, \\
  \eta_2(y,w) = (h^{-1}(y), h^{-1}(w)) \quad \text{ for } (y,w) \in G_\psi^a,
  \end{gather*}
 both the conditions (iii) and (iv) hold if and only if $h$ satisfies the condition (A).
  We thus obtain that 
 $(h, c_1, c_2, d_1, d_2)$ 
satisfies all of the conditions of Proposition \ref{prop:main1}
 if and only if $h$ satisfies the conditions (A) and (B) in Definition \ref{def:atc}. 
 \end{proof}
  Recall that two Smale spaces $(X,\phi) $ and $(Y,\psi)$ 
are asymptotically conjugate
if $(X,\phi) $ and $(Y,\psi)$ are asymptotically continuous orbit equivalent such that 
$c_1\equiv 1, c_2\equiv 1, d_1\equiv 0, d_2\equiv 0$.
 By Lemma \ref{lem:N6.4}
  we know that if   $(X,\phi) $ and $(Y,\psi)$ are asymptotically conjugate,
  then $h:X\longrightarrow Y$ satisfies the conditions (A) and (B) in Definition \ref{def:atc}.
  Conversely, if $h:X\longrightarrow Y$ satisfies the conditions (A) and (B),
  then  $(h, c_1, c_2, d_1, d_2)$ with $c_1\equiv 1, c_2\equiv 1, d_1\equiv 0, d_2\equiv 0$
  satisfies all of the conditions of Proposition \ref{prop:main1}.
  We thus have the following proposition.
\begin{proposition}\label{prop:acatc}
For Smale spaces $(X,\phi)$ and $(Y,\psi)$, the following two conditions are equivalent.
\begin{enumerate}
\renewcommand{\theenumi}{\roman{enumi}}
\renewcommand{\labelenumi}{\textup{(\theenumi)}}
\item 
$(X,\phi) $ and $(Y,\psi)$ are asymptotically conjugate.
\item 
$(X,\phi) $ and $(Y,\psi)$ are asymptotically topologically conjugate.
\end{enumerate}
\end{proposition}
  Therefore we have 
\begin{theorem}\label{thm:acatc}
Let $(X,\phi)$ and 
$(Y,\psi)$ are irreducible Smale spaces. 
The following conditions are equivalent:
\begin{enumerate}
\renewcommand{\theenumi}{\roman{enumi}}
\renewcommand{\labelenumi}{\textup{(\theenumi)}}
\item
$(X,\phi)$ and 
$(Y,\psi)$ are 
asymptotically topologically conjugate.
\item
There exists an isomorphism 
$\Phi:\R_\phi^a\longrightarrow\R_\psi^a$
of $C^*$-algebras such that 
$\Phi(C(X) )= C(Y)
$
and
$\Phi \circ \rho^\phi_t =\rho^\psi_{t}\circ \Phi$ for $t \in \T,$
\end{enumerate}
\end{theorem}
\begin{proof}
By \cite[Theorem 6.4]{MaCJM}, we know that 
$(X,\phi)$ and $(Y,\psi)$ are asymptotically conjugate if and only if 
the condition (ii) holds.
Therefore by Proposition \ref{prop:acatc},
we obtain the desired assertion.
\end{proof}

\section{Flip conjugacy and asymptotic continuous orbit equivalence}
Let $(X,\phi)$ be an irreducible Smale space.
Recall that for $p \in \Z$ with $p\ne 0$ 
a point $x \in X$ is called an asymptotic $p$-periodic point 
if $(\phi^p(x), x) \in G_\phi^a.$   

\begin{lemma}[{\cite[Lemma 5.3]{PutSp}}]
Let $x\in X$ be an asymptotic $p$-periodic point for some $p \in \Z$ with $p\ne 0.$ 
\begin{enumerate}
\renewcommand{\theenumi}{\roman{enumi}}
\renewcommand{\labelenumi}{\textup{(\theenumi)}}
\item 
The limit ${\displaystyle \lim_{k\to{\infty}}}\phi^{|p|k}(x)$ exists in $X$, 
denoted by $\eta^s(x),$ such that $\phi^p(\eta^s(x)) = \eta^s(x).$
\item 
The limit ${\displaystyle \lim_{k\to{\infty}}}\phi^{-|p|k}(x)$ 
exists in $X$, 
denoted by $\eta^u(x),$ such that $\phi^p(\eta^u(x)) = \eta^u(x).$
\end{enumerate}
\end{lemma}
\begin{proof}
(i) Since $(\phi^p(x), x) \in G_\phi^a$ and hence
$(\phi^p(x), x) \in G_\phi^s$ we have
$\phi^p(x) \in X^s(x)$.
By \cite[Lemma 5.3]{PutSp}, we know that 
the limit ${\displaystyle \lim_{k\to{\infty}}}\phi^{|p|k}(x),$ denoted by $\eta^s(x)$, 
exists in $X$
such that $\phi^p(\eta^s(x)) = \eta^s(x).$

(ii) is similarly shown to (i).
\end{proof}

\begin{lemma}\label{lem:6.2}
Assume that $(X,\phi) \ACOE (Y,\psi)$ via a homeomorphism $h:X\longrightarrow Y.$ 
  Let $x$ be an asymptotic  periodic point in $X.$
 \begin{enumerate}
\renewcommand{\theenumi}{\roman{enumi}}
\renewcommand{\labelenumi}{\textup{(\theenumi)}}
\item 
Both $h(x)$ and $h(\phi(x))$ are asymptotic periodic points in $(Y,\psi)$
suh that 
\begin{equation*}
\eta^s(h(\phi(x))) = \psi^{c_1(x)}(\eta^s(h(x))), \qquad
\eta^u(h(\phi(x))) = \psi^{c_1(x)}(\eta^u(h(x))).
\end{equation*}
\item
If $x$ is a $p$-periodic point, then 
both 
$\eta^s(h(x))$ and $\eta^u(h(x))$
are $c_1^p(x)$-periodic points, that is,
\begin{equation*}
\eta^s(h(x)) = \psi^{c_1^p(x)}(\eta^s(h(x))), \qquad
\eta^u(h(x)) = \psi^{c_1^p(x)}(\eta^u(h(x))).
\end{equation*}
\item If in particular  $x,$ 
$h(x)$ and $ h(\phi(x))$ are all periodic points, then we have 
\begin{equation*}
\eta^s(h(x)) =\eta^u(h(x)) = h(x), \qquad
\eta^s(h(\phi(x))) =\eta^u(h(\phi(x))) = h(\phi(x)) 
\end{equation*}
and hence
\begin{equation*}
h(\phi(x)) = \psi^{c_1(x)}(h(x)).
\end{equation*}
\end{enumerate}
\end{lemma}
\begin{proof}
(i) 
Assume that
$x$ is an asymptotic  $p$-periodic point in $X.$
Put $c_h^p(x) = c_1^p(x) + d_1(\phi^p(x), x).$
By \cite[Lemma 4.4]{MaCJM},
$h(x)$ is an asymptotic $c_h^p(x)$-periodic point in $Y$ 
so that both the limits
$\eta^s(h(x)) = {\displaystyle \lim_{k\to{\infty}}}\psi^{|c_h^p(x)|k}(h(x))
$
and
$\eta^u(h(x)) = {\displaystyle \lim_{k\to{\infty}}}\psi^{-|c_h^p(x)|k}(h(x))
$
exist in $Y.$
By Proposition \ref{prop:main1} (i), we know 
$(\psi^{c_1(x)}(h(x)), h(\phi(x))) \in G_\psi^a,$
so that 
\begin{align*}
\eta^s(h(\phi(x))) 
=& \lim_{k\to{\infty}}\psi^{|c_h^p(x)|k}(h(\phi(x))) 
= \lim_{k\to{\infty}}\psi^{|c_h^p(x)|k}(\psi^{c_1(x)}(h(x))) \\
=& \psi^{c_1(x)}(\lim_{k\to{\infty}}\psi^{|c_h^p(x)|k}(h(x))) 
= \psi^{c_1(x)}(\eta^s(h(x))).
\end{align*}
Similarly we have
$\eta^u(h(\phi(x))) = \psi^{c_1(x)}(\eta^u(h(x))).$

(ii)
We further assume that  $x$ is a $p$-periodic point.
Since $\phi^p(x) = x$, we see that $c_h^p(x) = c_1^p(x)$, so that 
$h(x)$ is an asymptotic $c_1^p(x)$-periodic point in $Y.$
Hence we have 
\begin{align*}
\eta^s(h(x)) 
=& \lim_{k\to{\infty}}\psi^{|c_1^p(x)|k}(h(x)) 
= \lim_{k\to{\infty}}\psi^{|c_1^p(x)|k}(\psi^{c_1^p(x)}(h(x))) \\
=& \psi^{c_1^p(x)}(\lim_{k\to{\infty}}\psi^{|c_1^p(x)|k}(h(x))) 
= \psi^{c_1^p(x)}(\eta^s(h(x))).
\end{align*}
Similarly we have
$\eta^u(h(x)) = \psi^{c_1^p(x)}(\eta^u(h(x))).$

(iii)
Assume that $x$ is a $p$-periodic point and 
$h(x), h(\phi(x))$ are also  periodic points in $(Y,\psi).$
Suppose that $\psi^q(h(x)) = h(x)$ for some $q\in \N, q\ne 0$.
As in the proof of (ii),
$h(x)$ is an asymptotic $c_1^p(x)$-periodic point.
Hence we hvae
$$
\eta^s(h(x)) = \lim_{k\to{\infty}}\psi^{|c_1^p(x)|k}(h(x))
= \lim_{m\to{\infty}}\psi^{|c_1^p(x)|mq}(h(x))
= h(x),
$$
and similarly 
$\eta^s(h(\phi(x)))= h(\phi(x)).$
By (i), we have
$h(\phi(x)) =\psi^{c_1(x)}(h(x)).  
$
\end{proof}
A homeomorphism $h:X\longrightarrow Y$ is said to be 
periodic point preserving if $h(x)\in Y$ is a periodic point of $(Y,\psi)$
for any periodic point $x \in X$ of $(X,\phi)$.
\begin{lemma}\label{lem:6.3}
Assume that Smale spaces
$(X,\phi)$ and $(Y,\psi)$ are irreducible.
Let  $h:X\longrightarrow Y$ 
be  a homeomorphism that gives rise to an asymptotic continuous orbit equivalence between $(X,\phi)$ and $(Y,\psi).$
Suppose that $h$ is periodic point preserving.
Then we have
\begin{equation}
h(\phi(x)) = \psi^{c_1(x)}(h(x)) \quad \text{ for all } x \in X. \label{eq:hphixallx}
\end{equation}
\end{lemma}
\begin{proof}
By Lemma \ref{lem:6.2}, 
the equality \eqref{eq:hphixallx} holds  for all periodic points $x\in X.$
Now $X$ is irreducible, so that  the set of periodic points 
is dense in $X$.
As  
 $c_1:X\longrightarrow \Z$ is continuous,
we have the equality \eqref{eq:hphixallx}
for all points in $X.$
\end{proof}
We have the following theorem.
\begin{theorem}\label{thm:flipacoe}
Let $(X,\phi)$ and $(Y,\psi)$ be irreducible Smale spaces.
The following conditions are equivalent:
\begin{enumerate}
\renewcommand{\theenumi}{\roman{enumi}}
\renewcommand{\labelenumi}{\textup{(\theenumi)}}
\item
$(X,\phi)$ and $(Y,\psi)$ are  flip  conjugate.
\item 
There exists a periodic point preserving homeomorphism
 $h:X\longrightarrow Y$ 
 that gives rise to an asymptotic continuous orbit equivalence between $(X,\phi)$ and $(Y,\psi).$
\end{enumerate}
\end{theorem}
\begin{proof}
The implication (i) $\Longrightarrow$ (ii) is obvious.
It remains to show that 
the implication (ii) $\Longrightarrow$ (i) holds.
Let  $h:X\longrightarrow Y$ be a periodic point preserving homeomorphism
  that gives rise to an asymptotic continuous orbit equivalence between $(X,\phi)$ and $(Y,\psi).$
By Lemma \ref{lem:6.3}, the equality  \eqref{eq:hphixallx}
hold for all points in $X.$
Since the cocycle function $c_1:X\longrightarrow \Z$ is continuous,
the equality  \eqref{eq:hphixallx} implies that 
the homeomorphisms
$\phi$ on $X$ and $\psi$ on $Y$
 are continuously orbit equivalent in the sense of Boyle-Tomiyama \cite{BT}.
By Boyle-Tomiyama's theorem  \cite[Theorem 3.2]{BT},
we conclude that $(X,\phi)$ and $(Y,\psi)$ are flip conjugate.
\end{proof}
There is no known examples of Smale spaces such that they are 
 asymptotically continuous orbit equivalent  but not flip conjugate.

\section{Asymptotic flip conjugacy and Ruelle algebras}
In this section, we introduce a notion of asymptotic flip conjugacy,
 that is an asymptotic version of flip conjugacy. 
\begin{definition}\label{def:StUniACOE}
Two Smale spaces 
$(X,\phi)$ and $(Y,\phi)$ are said to be {\it asymptotically flip conjugate}\/ 
if 
$(X,\phi)$ is asymptotically topologically conjugate to $(Y,\phi)$
or to $(Y,\psi)$.
\end{definition}
Since asymptotic topological conjugacy is equivalent to asymptotic conjugacy,
we have 
\begin{lemma}\label{lem:flipaflip}
Two Smale spaces $(X,\phi)$ and $(Y,\psi)$ 
are asymptotically flip conjugate if and only if  
they are asymptotically continuous orbit equivalent having
 their cocycle functions
$c_1, c_2$ and $d_1, d_2$ such as 

(1)
$
c_1\equiv 1,\, \,
c_2\equiv 1  \, \, \text{ and }
\, \,
d_1\equiv 0,\, \,
d_2\equiv 0, \, \,
$ or

(2) 
$
c_1\equiv -1,\, \,
c_2\equiv -1  \, \, \text{ and }
\, \, 
d_1\equiv 0, \, \,
d_2\equiv 0.
$
\end{lemma}
Let $(X,\phi)$ be an irreducible Smale space.
Now let us recall that the Ruelle algebra $\R_\phi^a$
 is defined by the $C^*$-algebra 
$C^*(G_\phi^a\rtimes\Z)$ of the amenable \'etale groupoid 
$G_\phi^a\rtimes\Z.$
As
$C^*(G_\phi^a\rtimes\Z)$
is isomorphic to the crossed product 
$C^*$-algebra
$C^*(G_\phi^a)\rtimes\Z$, 
it has a canonical action of $\T$ called the dual action written 
$\rho^\phi_t, t \in \T.$
Since the groupoid 
$G_\phi^a\rtimes\Z$ is essentially principal (\cite[Lemma 5.2, Lemma 5.3]{MaCJM}),
the abelian  $C^*$-algebra $C((G_\phi^a\rtimes\Z)^{(0)})$
of the unit space $(G_\phi^a\rtimes\Z)^{(0)}$ of the groupoid 
$G_\phi^a\rtimes\Z$ is a maximal abelian $C^*$-subalgebra of 
$C^*(G_\phi^a\rtimes\Z).$
As the space $X$ is identified with $(G_\phi^a\rtimes\Z)^{(0)},$
we may regard the abelian $C^*$-algebra $C(X)$ as 
the maximal abelian $C^*$-subalgebra $C^*((G_\phi^a\rtimes\Z)^{(0)})$ of $\R_\phi^a.$
We note the following proposition.
\begin{proposition}\label{prop:Ruelleflip}
The isomorphism
$(x,n,z) \in G_\phi^a\rtimes\Z \longrightarrow 
(x,-n, z) \in G_{\phi^{-1}}^a\rtimes\Z
$
of \'etale groupoids between
$G_\phi^a\rtimes\Z $ and $G_{\phi^{-1}}^a\rtimes\Z $
induces an isomorphism 
$\Phi:\R_\phi^a \longrightarrow\R_{\phi^{-1}}^a$
of $C^*$-algebras such that 
$\Phi(C(X)) = C(X)
$
and
$\Phi \circ \rho^\phi_t =\rho^{\phi^{-1}}_{-t}\circ \Phi$ for $t \in \T.$
\end{proposition}
\begin{proof}
As in \cite[Section 5]{MaCJM}, 
let us represent the groupoid $C^*$-algebras 
$\R_\phi^a$ and $\R_{\phi^{-1}}^a$
on the Hilbert $C^*$-right $C(X)$-modules
$l^2(G_\phi^a\rtimes\Z)$ and $l^2(G_{\phi^{-1}}^a\rtimes\Z )$, respectively.  
Let us denote by
$\varphi: G_\phi^a\rtimes\Z \longrightarrow  G_{\phi^{-1}}^a\rtimes\Z$
the isomorphism of the \'etale groupoids defined by
$\varphi(x,n,z) = (x,-n,z)$ for $(x,n,z) \in G_\phi^a\rtimes\Z.$
Define the homomorphisms
$f:G_\phi^a\rtimes\Z \longrightarrow \Z$
and
$g:G_{\phi^{-1}}^a\rtimes\Z\longrightarrow \Z$ of \'etale groupoids 
by
$f(x,n,z) = n$ and $g(x,n,z) = -n$ respectively,
so that we have
$f = g\circ \varphi.$
By \cite[Proposition 5.6]{MaCJM},
there exists
an isomorphism 
$\Phi:\R_\phi^a \longrightarrow\R_{\phi^{-1}}^a$
of $C^*$-algebras such that 
$\Phi(C(X)) = C(X)
$
and
$\Phi \circ \Ad(U_t(f)) = \Ad(U_t(g)) \circ \Phi$ for $t \in \T,$
where
$U_t(f)$ is a unitary on $l^2(G_\phi^a\rtimes\Z)$
defined by
\begin{equation*}
[U_t(f)\xi](x,n,z) 
= 
\exp(2\pi\sqrt{-1} f(x,n,z)t )\xi(x,n,z),\qquad
\xi \in l^2(G_\phi^a\rtimes\Z),
\end{equation*}
and 
$U_t(g)$ is similarly defined.
It is easy to see that 
 $\Ad(U_t(f)) =\rho^\phi_t$
and
$\Ad(U_t(g)) =\rho^{\phi^{-1}}_{-t}$
 for $t \in \T.$
Therefore we get 
$\Phi \circ \rho^\phi_t =\rho^{\phi^{-1}}_{-t}\circ \Phi$ for $t \in \T.$
\end{proof}

We thus obtain the following characterization of asymptotic flip conjugacy 
of Smale spaces.
\begin{theorem}[cf. {\cite[Lemma 6.2 and Theorem 6.4]{MaCJM}}]\label{thm:asmpflip}
Let
$(X,\phi)$ and $(Y,\psi)$ be irreducible Smale spaces. 
Then the following assertions are equivalent for $\varepsilon =\pm 1.$
\begin{enumerate}
\renewcommand{\theenumi}{\roman{enumi}}
\renewcommand{\labelenumi}{\textup{(\theenumi)}}
\item
$(X,\phi)$ and 
$(Y,\psi)$ are 
asymptotically  flip conjugate.
\item
There exists an isomorphism
$\varphi: G_\phi^a\rtimes\Z\longrightarrow G_\psi^a\rtimes\Z$
 of \'etale groupoids such that 
$d_\psi \circ \varphi = \varepsilon d_\phi.$ 
\item
There exists an isomorphism
$\varphi: G_\phi^a\rtimes\Z\longrightarrow G_\psi^a\rtimes\Z$
 of \'etale groupoids such that 
$\varphi(G_\phi^a) = G_\psi^a.$
\item
There exists an isomorphism 
$\Phi:\R_\phi^a \longrightarrow\R_\psi^a$
of $C^*$-algebras such that 
$\Phi(C(X)) = C(Y)
$
and
$\Phi \circ \rho^\phi_t =\rho^\psi_{\varepsilon t}\circ \Phi$ for $t \in \T.$
\item
There exists an isomorphism 
$\Phi:\R_\phi^a\longrightarrow\R_\psi^a$
of $C^*$-algebras such that 
$\Phi(C(X)) = C(Y)
$
and
$\Phi(C^*(G_\phi^a)) = C^*(G_\psi^a).$
\end{enumerate}  
\end{theorem}
\begin{proof}

(i) $\Longrightarrow$ (ii): 
If $(X,\phi)$ is asymptotically conjugate to $(Y,\psi),$
then 
there exists an isomorphism
$\varphi: G_\phi^a\rtimes\Z\longrightarrow G_\psi^a\rtimes\Z$
 of \'etale groupoids such that 
$d_\psi \circ \varphi = d_\phi$ 
by \cite[Theorem 6.4]{MaCJM}.
If $(X,\phi)$
is asymptotically  conjugate to $(Y,\psi^{-1}),$
then one may take the above isomorphism
$\varphi: G_\phi^a\rtimes\Z\longrightarrow G_\psi^a\rtimes\Z$
 of \'etale groupoids such that 
$d_\psi \circ \varphi = -d_\phi.$ 

(ii) $\Longrightarrow$ (i):
Assume the assertion (ii) for $\varepsilon = 1$.
Then by \cite[Theorem 6.4]{MaCJM}, 
we know that  $(X,\phi)$ is asymptotically conjugate to $(Y,\psi).$
Assume next the assertion (ii) for $\varepsilon = -1$.
Then  $(X,\phi)$ is asymptotically conjugate to $(Y,\psi^{-1}).$

(i) $\Longleftrightarrow$ (iv): 
The assertions follow from \cite[Theorem 6.4]{MaCJM} 
together with Proposition \ref{prop:Ruelleflip}.

(iii) $\Longleftrightarrow$ (v): 
The assertions follow from \cite[Lemma 6.2]{MaCJM}. 

(ii) $\Longrightarrow$ (iii):
Suppose that 
there exists an isomorphism
$\varphi: G_\phi^a\rtimes\Z\longrightarrow G_\psi^a\rtimes\Z$
 of \'etale groupoids such that 
$d_\psi \circ \varphi = \varepsilon d_\phi.$ 
Hence the correspondence 
$
\varphi:
(x,n,z) \in G_\phi^a\rtimes\Z\longrightarrow 
(h(x), \varepsilon n, h(z)) \in G_\psi^a\rtimes\Z
$ 
gives rise to an isomorphism of \'etale groupoids.
We thus have 
$d_1 \equiv 0$ and $d_2\equiv 0$
so that $\varphi(G_\phi^a) = G_\psi^a.$

(iii) $\Longrightarrow$ (i):
Suppose that 
an isomorphism
$\varphi: G_\phi^a\rtimes\Z\longrightarrow G_\psi^a\rtimes\Z$
 of \'etale groupoids satisfies 
$\varphi(G_\phi^a) = G_\psi^a.$
By \cite[Theorem 3.4]{MaCJM},
we know that the restriction of $\varphi$ 
to its unit space $(G_\phi^a\rtimes\Z)^{(0)}$
gives rise to a homeomorphism $h : X\longrightarrow Y$ 
that gives rise to an asymptotic continuous orbit equivalence.
Let $c_1, c_2$ and $ d_1, d_2$ be their cocycle functions.
 The condition 
$\varphi(G_\phi^a) = G_\psi^a$
implies that $d_1 \equiv 0.$
By Proposition \ref{prop:main1} (1)
we have 
$c_1(x) = c_1(z)$
for all $(x,z) \in G_\phi^a.$ 
Hence the function $c_1: X\longrightarrow \Z$ 
is asymptotically invariant
so that it is a constant integer by {\cite[7.16 (b)]{Ruelle2}}, denoted by $K_1$.
By Proposition \ref{prop:main1} (vi),
we see that 
$c_2(y) K_1 =1$ for all $y \in Y$
so that 
$c_2(y)$ is also a constant integer, denoted by $K_2$,
such that 
$K_1\cdot K_2 =1.$
By Proposition \ref{prop:main1} (viii),
we have 
$c_1^{d_2(y,w)}(h^{-1}(y)) =0$ for all $(y, w) \in G_\psi^a,$
because $d_1(x,z)=0$ for all $(x,z) \in G_\phi^a.$
As
$c_1^{d_2(y,w)}(h^{-1}(y)) =d_2(y,w)K_1,$
we obtain that   
$d_2(y,w)=0$ for all $(y,w) \in G_\psi^a.$
This shows that the following two cases occur:

(1)
$
c_1\equiv 1,\, \,
c_2\equiv 1  \, \, \text{ and }
\, \,
d_1\equiv 0,\, \,
d_2\equiv 0, \, \,
$ or

(2) 
$
c_1\equiv -1,\, \,
c_2\equiv -1  \, \, \text{ and }
\, \, 
d_1\equiv 0, \, \,
d_2\equiv 0.
$

By Lemma \ref{lem:flipaflip},
we know that 
$(X,\phi)$ and 
$(Y,\psi^{-1})$ are 
asymptotically flip conjugate.
\end{proof}

Topological Markov shifts, often called shifts of finite type or SFT for brevity, 
form a basic class of Smale spaces.
Furthermore
it is well-known that 
any Smale space is realized as a finite factor of a topological Markov shift 
by its Markov partitions (\cite[Theorem 7.6]{Ruelle1}).
Let $A=[A(i,j)]_{i,j=1}^N$
be an $N\times N$ matrix with entries in $\{0,1\}.$
The two-sided shift space
$\bXA$ is defined by
\begin{equation*}
\bXA = \{ (x_i)_{i \in \Z} \mid A(x_i, x_{i+1}) =1 \text{ for all } i \in \Z\}.
\end{equation*}
 For a given real number $\lambda_0$ with $0<\lambda_0 <1,$
$\bXA$ becomes a compact metric space defined by the metric
\begin{equation*}
d((x_i)_{i\in \Z},(z_i)_{i\in \Z}) = \lambda_0^k
\quad
\text{ where } 
\quad
k = \inf\{|i| :  x_i \ne z_i \}
\end{equation*}
for $(x_i)_{i\in \Z}, \, (z_i)_{i\in \Z} \in \bar{X}_A$
with $(x_i)_{i\in \Z} \ne (z_i)_{i\in \Z}.$
The shift map
$\bsA:\bXA\longrightarrow \bXA$
is a homeomorphism defined by $\bsA((x_i)_{i\in \Z}) =(x_{i+1})_{i\in \Z}.$
The topological dynamical system
$(\bXA,\bsA)$ is called the (two-sided) topological Markov shift, or a shift of finite type (SFT),
defined by the matrix $A$.  
Let $A$ and $B$ be an $N\times N$ matrix with entries in $\{0,1\}$
and
an $M\times M$ matrix with entries in $\{0,1\},$
respectively.
We assume that both $A$ and $B$ are irreducible and not any paermutation.
We then see  the following proposition:
\begin{proposition}\label{prop:flipSFT}
Two-sided topological Markov shifts 
$(\bar{X}_A, \bar{\sigma}_A)$ and 
$(\bar{X}_B, \bar{\sigma}_B)$ 
are flip conjugate
if and only if
they are
asymptotically  flip conjugate.
\end{proposition}
\begin{proof}
By \cite[Proposition 11.1]{MaCJM}, 
flip conjugacy yields asymptotic flip conjugacy. 
Conversely suppose that 
$(\bar{X}_A, \bar{\sigma}_A)$ and 
$(\bar{X}_B, \bar{\sigma}_B)$ 
are asymptotically flip conjugate.
If $(\bar{X}_A, \bar{\sigma}_A)$ and 
$(\bar{X}_B, \bar{\sigma}_B)$ 
are asymptotically conjugate, 
then  they are topologically conjugate by \cite{MaPre2017a}.
If    $(\bar{X}_A, \bar{\sigma}_A)$ and 
$(\bar{X}_B, \bar{\sigma}_B^{-1})$
are asymptotically conjugate, 
then 
  $(\bar{X}_A, \bar{\sigma}_A)$ and 
$(\bar{X}_B, \bar{\sigma}_B^{-1})$
are topologically conjugate, so that 
$(\bar{X}_A, \bar{\sigma}_A)$ and 
$(\bar{X}_B, \bar{\sigma}_B)$ 
are flip conjugate.
\end{proof}
Let $A, B$ be irreducible non-permutation matrices with entries in $\{0,1\}.$
Let us denote by 
$G_A^a, G_B^a$
 the asymptotic \'etale groupoids
$G_{{\bar\sigma}_A}^a,G_{{\bar\sigma}_B}^a,$
respectively.
Let us also denote by $\R_A^a, \R_B^a$ the asymptotic Ruelle algebras 
$\R_{\bar{\sigma}_A}^a (=C^*(G_{A}^a\rtimes\Z)), 
\R_{\bar{\sigma}_B}^a (=C^*(G_{B}^a\rtimes\Z))
$ 
for the 
topological Markov shifts
$(\bar{X}_A, {\bar\sigma}_A), (\bar{X}_B, {\bar\sigma}_B),$
respectively.
We set their dual actions 
$\rho^A := \rho^{{\bar\sigma}_A}, \, \rho^B := \rho^{{\bar\sigma}_B},
$
respectively.
By Theorem \ref{thm:asmpflip} 
together with 
Proposition \ref{prop:flipSFT},
we may characterize
the flip conjugacy class of irreducible two-sided  topological Markov shifts in terms of
$C^*$-algebras in the following way. 
\begin{corollary}\label{cor:flipSFT}
Keep the above notation.
The following conditions are equivalent: 
\begin{enumerate}
\renewcommand{\theenumi}{\roman{enumi}}
\renewcommand{\labelenumi}{\textup{(\theenumi)}}
\item
$(\bar{X}_A, {\bar\sigma}_A)$ and $ (\bar{X}_B, {\bar\sigma}_B)$
are  flip conjugate.
\item
There exists an isomorphism 
$\Phi:\R_A^a\longrightarrow\R_B^a$
of $C^*$-algebras such that 
$\Phi(C(\bar{X}_A)) = C(\bar{X}_B)
$
and
$\Phi \circ \rho^A_t =\rho^B_{\varepsilon t}\circ \Phi$ for $t \in \T,$
where 
$\varepsilon = 1$ or $-1.$ 
\item
There exists an isomorphism 
$\Phi:\R_A^a\longrightarrow\R_B^a$
of $C^*$-algebras such that 
$\Phi(C(\bar{X}_A)) = C(\bar{X}_B)
$
and
$\Phi(C^*(G_A^a)) = C^*(G_B^a).$
\end{enumerate}
\end{corollary}
Hence the triplet
$(\R^a_A, C^*(G_A^a), C(\bar{X}_A))$
of $C^*$-subalgebras of $\R_A^a$ 
is a complete invariant of the flip conjugacy class of the two-sided topological Markov shift
$(\bar{X}_A, {\bar\sigma}_A)$.

\medskip

{\it Acknowledgments:}
This work was supported by JSPS KAKENHI 
Grant Numbers 15K04896, 19K03537.


\end{document}